\documentstyle[amssymb,amsfonts,12pt]{amsart}

\newtheorem{theorem}{Theorem}[section]
\newtheorem{lemma}[theorem]{Lemma}
\newtheorem{corollary}[theorem]{Corollary}
\newtheorem{proposition}[theorem]{Proposition}
\newtheorem{definition}[theorem]{Definition}

\newtheorem{remark}[theorem]{Remark}
\newtheorem*{ack*}{Acknowledgment}

\def\x{{\bf x}}
\def\T{{\mathbb T}}
\def\F{{\mathcal F}}
\def\R{{\mathbb R}}
\def\E{{\mathcal H}}
\def\N{{\mathbb N}}
\def\C{{\mathbb C}}

\def\Z{{\mathbb Z}}
\def\P{{\mathcal P}}
\def\B{{\mathbb E}}

\def\A{{\mathcal N}}
\def\dist{{\operatorname{dist}}}
\def\supp{{\operatorname{supp}}}
\def\bas{\begin{align*}}
\def\eas{\end{align*}}
\def\bi{\begin{itemize}}
\def\ei{\end{itemize}}
\newenvironment{proof}{\noindent {\bf Proof} }{\endprf\par}
\def \endprf{\hfill  {\vrule height6pt width6pt depth0pt}\medskip}
\def\1{{\bf 1}}

\begin{document}

\title[Decouplings for hypersurfaces and curves]{Decouplings for curves and hypersurfaces with nonzero Gaussian curvature}
\author{Jean Bourgain}
\address{School of Mathematics, Institute for Advanced Study, Princeton, NJ 08540}
\email{bourgain@@math.ias.edu}
\author{Ciprian Demeter}
\address{Department of Mathematics, Indiana University, 831 East 3rd St., Bloomington IN 47405}
\email{demeterc@@indiana.edu}

\thanks{The first author is partially supported by the NSF grant DMS-1301619. The second  author is partially supported  by the NSF Grant DMS-1161752}
\begin{abstract}
We prove two types of results. First we develop the  decoupling theory for hypersurfaces with nonzero Gaussian curvature, which extends our earlier work from \cite{BD3}. As a consequence of this we obtain sharp (up to $\epsilon$ losses) Strichartz estimates for the hyperbolic Schr\"odinger equation on the torus.

Our second main result is an $l^2$ decoupling for non degenerate curves which has implications for Vinogradov's mean value theorem.
\end{abstract}
\maketitle


\section{Statements of results}

Let $S$ be a compact $C^2$ hypersurface in $\R^n$  with  nonzero Gaussian curvature.  The typical example to have in mind is the truncated  paraboloid defined for $\upsilon=(\upsilon_1,\ldots,\upsilon_{n-1})\in(\R\setminus\{0\})^{n-1}$ as
$$H^{n-1}_\upsilon:=\{(\xi_1,\ldots,\xi_{n-1},\upsilon_1\xi_1^2+\ldots+\upsilon_{n-1}\xi_{n-1}^2):\;|\xi_i|\le 1/2\}.$$
 $H^{n-1}_\upsilon$ is called elliptic when all $\upsilon_i$ have the same sign and  hyperbolic otherwise.

Let $\A_\delta=\A_\delta(S)$ be the $\delta$ neighborhood of $S$ and let $\P_\delta$ be a finitely overlapping cover of  $\A_\delta$ with
$\sim\delta^{1/2}\times \ldots\delta^{1/2}\times \delta$ rectangular boxes $\theta$. We will denote by $f_\theta$ the Fourier restriction of $f$ to $\theta$.

We will write $A\sim B$ if $A\lesssim B$ and $B\lesssim A$. The implicit constants hidden inside the symbols  $\lesssim$ and $\sim$  will in general  depend on  fixed parameters such as $p$, $n$, $\alpha$ and sometimes on  variable parameters such as  $\epsilon$. We will in general not record the dependence on the fixed parameters.

Our first result is the following {\em $l^p$ decoupling theorem\footnote{Perhaps a more appropriate nomenclature for the $l^p$ and $l^2$ decouplings we consider in this paper would be $l^p(L^p)$ and $l^2(L^p)$ decouplings. For simplicity of notation,  we prefer the former notation.}}.
\begin{theorem}
\label{t1}

Let $n\ge 2$.  If $\supp(\hat{f})\subset \A_\delta$
then for $p\ge\frac{2(n+1)}{n-1}$ and $\epsilon>0$
\begin{equation}
\label{e6}
\|f\|_p\lesssim_\epsilon \delta^{\frac{n}p-\frac{n-1}{2}-\epsilon}(\sum_{\theta\in \P_\delta}\|f_\theta\|_p^p)^{1/p}.
\end{equation}
\end{theorem}
This is a close cousin of the following $l^2$ decoupling proved in \cite{BD3} in the case when $S$ has definite second fundamental form
\begin{equation}
\label{e41}
 \|f\|_p\lesssim_\epsilon \delta^{-\frac{n-1}4+\frac{n+1}{2p}-\epsilon}(\sum_{\theta\in \P_\delta}\|f_\theta\|_p^2)^{1/2},\;\;p\ge\frac{2(n+1)}{n-1}.
\end{equation}
We point out that \eqref{e6} is slightly weaker than \eqref{e41} as it follows from  \eqref{e41} via H\"older's inequality. In particular, Theorem \ref{t1} for $n=2$ is contained in \cite{BD3}. We mention that sharp $l^p$ decouplings were first considered by Wolff  in the case of the cone, see \cite{TWol}.

As briefly explained in \cite{BD3}, \eqref{e41} is false for the hyperbolic paraboloid due to the fact that it contains lines. On the other hand, apart from the dependence on $\epsilon$, inequality \eqref{e6} is sharp, as is \eqref{e41}. This can be easily seen by considering the case of the sphere $S=S^{n-1}$ and $f$ with $\widehat{f}=1_{\A_\delta}$.
 The main new difficulty in proving \eqref{e6} as compared to \eqref{e41} is the fact that intersections of hyperbolic paraboloids with hyperplanes do not always have nonzero Gaussian curvature. It will  however be crucial to our argument  that at most one of the principal curvatures of these sections can be small.
An application of Theorem \ref{t1} to curves is discussed in Section \ref{sec:finall}.
\bigskip

A  modification of our proof of Theorem \ref{t1} leads to the following related $l^2$ decoupling result for $H^{n-1}_\upsilon$. Let us denote by $d(\upsilon)$ the minimum between the number of positive and negative entries of $\upsilon$.

\begin{theorem}
\label{t1:l2notlp}
Let $n\ge 2$.  If $\supp(\hat{f})\subset \A_\delta(H^{n-1}_\upsilon)$
then for $p\ge 2$ and $\epsilon>0$ we have
\begin{equation}
\label{e6l2notlp}
\|f\|_p\lesssim_\epsilon \delta^{-\epsilon}K^{(2)}_{n,p,\upsilon}(\delta)(\sum_{\theta\in \P_\delta}\|f_\theta\|_p^2)^{1/2},
\end{equation}
where $K^{(2)}_{n,p,\upsilon}(\delta)=\delta^{-\frac{n-1}4+\frac{n+1}{2p}}$ when $p\ge \frac{2(n+1-d(\upsilon))}{n-1-d(\upsilon)}$ and $K^{(2)}_{n,p,\upsilon}(\delta)=\delta^{d(\upsilon)(-\frac14+\frac1{2p})}$ when $2\le p\le \frac{2(n+1-d(\upsilon))}{n-1-d(\upsilon)}$.
\end{theorem}

\bigskip

Note that Theorems \ref{t1} and  \ref{t1:l2notlp} are independent, neither of them implies the other one.  Of course, \eqref{e6l2notlp}  generalizes \eqref{e41} (which corresponds to  $d(\upsilon)=0$), and leads to Strichartz estimates for the hyperbolic Schr\"odinger equation. More precisely, fix $\upsilon_1,\dots,\upsilon_{n-1}\in \R\setminus \{0\}$. For $\phi\in L^2(\T^{n-1})$  consider the "generalized Laplacian" operator
$$T \phi(x_1,\ldots,x_{n-1})=$$
$$\sum_{(\xi_1,\ldots,\xi_{n-1})\in\Z^{n-1}}(\xi_1^2\upsilon_1+\ldots+\xi_{n-1}^2\upsilon_{n-1})\hat{\phi}(\xi_1,\ldots,\xi_{n-1})e(\xi_1x_1+\ldots+\xi_{n-1}x_{n-1})$$
on the irrational torus $\prod_{i=1}^{n-1}\R/(|\upsilon_i|\Z)$. Let also
$$e^{itT}\phi(x_1,\ldots,x_{n-1},t)=$$$$\sum_{(\xi_1,\ldots,\xi_{n-1})\in\Z^{n-1}}\hat{\phi}(\xi_1,\ldots,\xi_{n-1})e(x_1\xi_1+\ldots+x_{n-1}\xi_{n-1}+t(\xi_1^2\upsilon_1+\ldots+\xi_{n-1}^2\upsilon_{n-1})).$$
Following the approach described in \cite{BD3}, Theorem \ref{t1:l2notlp} implies the next corollary.

\begin{corollary}[Strichartz estimates for irrational tori: the hyperbolic case]
\label{thmmmm3} Let $\phi\in L^2(\T^{n-1})$ with $\supp(\hat{\phi})\subset [-N,N]^{n-1}$.
Then for each $\epsilon>0$, $p\ge 2$ and each interval $I\subset\R$ with $|I|\gtrsim 1$ we have
\begin{equation}
\label{EE52}
\|e^{itT}\phi\|_{L^{p}(\T^{n-1}\times I)}\lesssim_{\epsilon} N^{\epsilon}K^{(2)}_{n,p,\upsilon}(N^{-2})|I|^{1/p}\|\phi\|_2,
\end{equation}
where $K^{(2)}_{n,p,\upsilon}$ is as in Theorem \ref{t1:l2notlp} and the implicit constant does not depend on $I$ and  $N$.
\end{corollary}

Estimates of this type have been considered recently, see for example \cite{GT} and \cite{Wa}. It is interesting to note that in the  non elliptic case ($d(\upsilon)\ge 1$), the exponent of $N$ in $K^{(2)}_{n,p,\upsilon}(N^{-2})$ is always nonzero, in contrast with the continuous case (when $\T^n$ is replaced with $\R^n$) and the elliptic case (on either $\T^{n}$ or $\R^{n}$). This exponent is sharp  in the case when $|\upsilon_i|=1$ for all $i$. To see this, use $\phi$ Fourier supported on the lattice points of a vector subspace of dimension $d(\upsilon)$ of $H^{n-1}_{\upsilon}$. This example does not exist in the case when $\upsilon$ has rationally independent entries. Our method here does not seem to shed any light on the issue of  whether one can improve the constant $K^{(2)}_{n,p,\upsilon}(N^{-2})$ in \eqref{EE52} in the irrational case.

It is possible that the term $N^\epsilon$ is not necessary in \eqref{EE52} when $p>\frac{2(n+1-d(\upsilon))}{n-1-d(\upsilon)}$. As observed in \cite{BD3}, this is indeed the case for the elliptic Schr\"odinger equation on the rational torus. See also  \cite{GT} and \cite{Wa} where similar sharp results are proved in the range  $2\le p\le 4$, when $n=3$.

The proof of Theorem \ref{t1:l2notlp} will be sketched in Section \ref{sec:Str}.

\bigskip

In the second part of the paper we consider curves $\Phi:[0,1]\to\R^n$, $$\Phi(t)=(\phi_1(t),\ldots,\phi_{n}(t))$$ with $\phi_i\in C^n([0,1])$ and such that the Wronskian
$$W(\phi_1',\ldots,\phi_{n}')(t)$$
is nonzero on $[0,1]$. These  are usually referred to in literature as {\em nondegenerate} curves.

Abusing earlier notation, let $\A_\delta$ be the $\delta$ neighborhood of $\Phi([0,1])$ and let $\P_\delta$ be the partition of  $\A_\delta$ with $\delta$ neighborhoods $\theta$ of $\Phi(I)$, with $I$ dyadic interval of length $\delta^{1/n}$.
We will as before denote by $f_\theta$ the Fourier restriction of $f$ to $\theta$. One important aspect to note at this point is the fact that $\theta$ is a curved tube, not a straight one. To make $\theta$ a straight tube, one needs to consider intervals $I$ of length $\delta^{1/2}$.

\begin{theorem}
\label{t2}
For each such curve $\Phi$ and each $f:\R^n\to\C$ with Fourier support in $\A_\delta$ we have
$$\|f\|_{L^{p}(\R^n)}\lesssim_\epsilon\delta^{-\epsilon}(\sum_{\theta\in \P_\delta}\|f_\theta\|_{L^{p}(\R^n)}^{2})^{\frac12},$$
for each $\epsilon>0$ and $2\le p\le 4n-2.$
\end{theorem}

By using a limiting procedure (see for example the discussion in \cite{BD3} and \eqref{e11} below) one can replace $f$ with a sum of Dirac deltas and derive the following corollary.
\begin{corollary}
\label{corrrrr771}
Fix $\Phi$ as above and let $p\le 4n-2$. Then for each $\delta$-separated set $\Lambda$ of points on the curve $\Phi$ and each coefficients $a_\xi\in \C$ we have
\begin{equation}
\label{ropig90580vyunmv45r0-c489rt870-910-}
(\frac{1}{|B_R|}\int_{B_R}|\sum_{\xi\in\Lambda}a_\xi e(\xi\cdot x)|^pdx)^{1/p}\lesssim_\epsilon \delta^{-\epsilon}\|a_\xi\|_{l^2(\Lambda)},
\end{equation}
for each $\epsilon$ and each ball $B_R\subset \R^n$ of radius $R\gtrsim \delta^{-n}$.
\end{corollary}
\bigskip

In particular, by applying this to (a rescaled version of) the curve
$$\Phi(t)=(t,t^2,\ldots,t^n)$$ we get
\begin{equation}
\label{ropig90580vyunmv45r0-c489rt870-910-later}
(\frac{1}{|B_R|}\int_{B_R}|\sum_{\xi\in\Lambda}a_\xi e(\xi\cdot x)|^pdx)^{1/p}\lesssim_\epsilon N^{-\epsilon}\|a_\xi\|_{l^2(\Lambda)},
\end{equation}
for $p\le 4n-2$, each 1-separated set $\Lambda$ of points on the curve $\Phi$ ($0\le t\le \infty$) with size $N$ and each $R\gtrsim N^n$.

Let us now explore an immediate consequence. Given an integer $k\ge 2$ we define the  $k$-energy of $\Lambda$ as
$$\B_k(\Lambda)=|\{(\lambda_1,\ldots,\lambda_{2k})\in \Lambda^{2k}:\;\lambda_1+\ldots+\lambda_k=\lambda_{k+1}+\ldots+\lambda_{2k}\}|.$$
By letting $R\to\infty$ in \eqref{ropig90580vyunmv45r0-c489rt870-910-later} with $p=4n-2$ and $a_\xi=1$, we immediately get that $$\B_{2n-1}(\Lambda)\lesssim_\epsilon \Lambda^{2n-1+\epsilon}$$  for each $\Lambda$ as above. In particular, we have
$$\B_{2n-1}(\{(l,l^2,\ldots,l^n):\;l=1,2,\ldots,N\})\lesssim_\epsilon N^{2n-1+\epsilon}.$$
This is a special case of Vinogradov's mean value theorem, due to (and significantly improved by the recent work of) Wooley, see for example \cite{Woo} and the references therein.
Our method however shows that the integer case is not special, but is rather a particular case of a larger phenomenon. Note also that our method allows for arbitrary coefficients $a_\xi$ in \eqref{ropig90580vyunmv45r0-c489rt870-910-later}.

Vinogradov's mean value theorem conjectures that \eqref{ropig90580vyunmv45r0-c489rt870-910-later} should hold in the integer case, with $a_\xi=1$, for $p$ as large as $n(n+1)$. It is possible that  \eqref{ropig90580vyunmv45r0-c489rt870-910-later} also holds for $p\le n(n+1)$, for arbitrary coefficients and frequency points $\Lambda$. However, our method does not currently seem to shed any light on this issue.
Further applications of variants of inequality \eqref{ropig90580vyunmv45r0-c489rt870-910-later} to number theory are presented in \cite{Bo}.

\begin{ack*}

The second author is indebted to Andrea Nahmod for bringing the question about the hyperbolic Schr\"odinger equation to his attention and for stimulating discussions on the topic.

\end{ack*}

\section{$l^p$ decouplings for hypersurfaces}
\bigskip

In this section we present the proof of Theorem \ref{t1}. We start by observing that the induction on scales argument from the last section in \cite{BD3} allows us to focus on the hypersurfaces $H^{n-1}_\upsilon$.

The proof of Theorem \ref{t1} for $H^{n-1}_\upsilon$  will be done in two separate stages. First, we develop the multilinear decoupling theory and show that it is essentially equivalent to its linear counterpart. Then we finish the proof by using a bootstrapping argument that relies on the equivalence between the linear and multilinear decoupling.

\subsection{Linear and multilinear decoupling}
\medskip

Let $g:H^{n-1}_\upsilon\to\C$. For a cap $\tau$ on $H^{n-1}_\upsilon$ we let $g_\tau=g1_\tau$ be the (spatial) restriction of $g$ to $\tau$.
We denote by $\pi:H^{n-1}_\upsilon\to [-1/2,1/2]^{n-1}$ the projection map and by $d\sigma$ the natural surface measure on $H^{n-1}_\upsilon$.
\begin{definition}
We say that the caps $\tau_1,\ldots,\tau_n$ on $H^{n-1}_\upsilon$ are $\nu$-transverse if the volume of the parallelepiped spanned by any unit normals $n_i$ at $\tau_i$ is greater than $\nu$.
\end{definition}

In the following, the norm $\|f\|_{L^p(w_{B_R})}$ will refer to the weighted $L^p$ integral $$(\int_{\R^n}|f(x)|^pw_{B_R}(x)dx)^{1/p}$$
for some positive weight $w_{B_R}$ which is Fourier supported in $B(0,\frac{1}{R})$ and satisfies
\begin{equation}
\label{wnret78u-0943mt7-w,-,1ir8gnrnfyqerbtf879wumweryermf,}
1_{B_R}(x)\lesssim w_{B_R}(x)\le(1+\frac{|x-c(B_R)|}{R})^{-10n}.
\end{equation}

We denote by $C_{p,n,\upsilon}(\delta,\nu)$ the smallest constant such that
$$\|(\prod_{i=1}^n|\widehat{g_{\tau_i}d\sigma}|)^{1/n}\|_{L^p(B_{\delta^{-1}})}\le C_{p,n,\upsilon}(\delta,\nu)\left[\prod_{i=1}^n(\sum_{\theta:\;\delta^{1/2}-\text{cap}\atop{\theta\subset\tau_i}}\|\widehat{g_{\theta}d\sigma}\|_{L^p(w_{B_{\delta^{-1}}})}^p)^{1/p}\right]^{1/n},$$
for each $\nu$-transverse caps $\tau_i\subset H^{n-1}_\upsilon$, each $\delta^{-1}$ ball $B_{\delta^{-1}}$ and each $g:H^{n-1}_\upsilon\to\C$.

Let also $K_{p,n,\upsilon}(\delta)$ be the smallest constant such that
$$\|\widehat{gd\sigma}\|_{L^p(B_{\delta^{-1}})}\le K_{p,n,\upsilon}(\delta)(\sum_{\theta:\delta^{1/2}-\text{cap}}\|\widehat{g_{\theta}d\sigma}\|_{L^p(w_{B_{\delta^{-1}}})}^p)^{1/p},$$
for each $g:H^{n-1}_\upsilon\to\C$ and each $\delta^{-1}$ ball $B_{\delta^{-1}}$.
\bigskip

In order to prove Theorem \ref{t1} it suffices to prove that $K_{p,n,\upsilon}(\delta)\lesssim_\epsilon\delta^{\frac{n}p-\frac{n-1}{2}-\epsilon}$ for $p\ge\frac{2(n+1)}{n-1}$.
In fact, if we denote by $K_{p,n,\upsilon}^{(1)}(\delta)$  the smallest constant such that
$$\|f\|_p\le K_{p,n,\upsilon}^{(1)}(\delta)(\sum_{\theta\in \P_\delta}\|f_\theta\|_p^p)^{1/p}$$
then we have $K_{p,n,\upsilon}^{(1)}(\delta)\sim K_{p,n,\upsilon}(\delta)$. We give a brief sketch on why this is the case and leave the details to the interested reader. The precise argument can be carried out using mollifications (see for example the proof of Lemma 2.2 in \cite{BCT}).

We start with the inequality $K_{p,n,\upsilon}(\delta)\lesssim K_{p,n,\upsilon}^{(1)}(\delta)$. It suffices to note that
$$\|\widehat{gd\sigma}\|_{L^p(B_{\delta^{-1}})}\lesssim \|\widehat{gd\sigma}\|_{L^p(w_{B_{\delta^{-1}}})}$$
and that $f:=(\widehat{gd\sigma})w_{B_{\delta^{-1}}}$ has the Fourier transform supported in $\A_\delta$.

To see the inequality $K_{p,n,\upsilon}^{(1)}(\delta)\lesssim K_{p,n,\upsilon}(\delta)$, first note that it suffices to prove that for  $f$ as in Theorem \ref{t1} and for $B$ the ball centered at the origin with radius $\delta^{-1}$ we have
$$(\int|fv_B|^p)^{1/p}\lesssim K_{p,n,\upsilon}(\delta)(\sum_{\theta\in \P_\delta}\|f_\theta\|_{L^p(w_B)}^p)^{1/p}.$$
Here $v_B$ is an appropriate weight with the same properties as $w_B$.
Indeed, once this is established, by translation invariance it will hold on each ball with radius $\delta^{-1}$. Then raise to the power $p$ and sum over an appropriate family of balls. Next, note that $\widehat{fv_B}=\widehat{f}*\widehat{v_B}$ is essentially constant at scale $\delta$. We can thus assume from the start that $\widehat{f}$ is essentially constant at scale $\delta$.  To eliminate irrelevant technicalities, we will assume that for some $g:H^{n-1}_\upsilon\to\C$ we have
$$\widehat{f}(\xi)=g(\eta)$$
for each $\xi=\eta+te_n\in\A_\delta$ with $\eta\in H^{n-1}_\upsilon$ and $|t|\le\delta$. In other words, we only assume that $\widehat{f}$ is constant vertically.
It follows that $$f(x)=2\delta\widehat{gd\sigma}(x)$$
and
$$f_\theta(x)=2\delta\widehat{g_\theta d\sigma}(x),$$
where $g_\theta$ is the restriction of $g$ to the projection of $\theta$ onto $H^{n-1}_\upsilon$.
Finally,
$$\|f\|_{L^p(B)}=2\delta\|\widehat{gd\sigma}\|_{L^p(B)}\le 2\delta K_{p,n,\upsilon}(\delta)(\sum_{\theta:\delta^{1/2}-\text{cap}}\|\widehat{g_{\theta}d\sigma}\|_{L^p(w_{B_{\delta^{-1}}})}^p)^{1/p}=$$
$$=K_{p,n,\upsilon}(\delta)(\sum_{\theta\in \P_\delta}\|f_\theta\|_{L^p(w_{B_{\delta^{-1}}})}^p)^{1/p}$$
\bigskip

\bigskip

Returning to the linear and multilinear constants, H\"older's inequality gives
$$C_{p,n,\upsilon}(\delta,\nu)\le K_{p,n,\upsilon}(\delta).$$
We will show that the reverse inequality essentially holds true.
\begin{theorem}
\label{t4}
Fix $n\ge 3$, $\upsilon\in\{-1,1\}^{n-1}$ and let $p\ge 2$. Assume one of the following holds

(i) $n=3$

(ii) $n\ge 4$ and
\begin{equation}
\label{e42}
K_{p,n-2,\upsilon'}(\delta')\lesssim_\epsilon {\delta'}^{-\frac{n-3}{2}(\frac12-\frac1p)-\epsilon}
\end{equation}
 for each $\delta'>0$, $\upsilon'\in\{-1,1\}^{n-3}$ and  each $\epsilon>0$.

Then for each $0<\nu\le 1$  there is $\epsilon(\nu)$ with $\lim_{\nu\to 0}\epsilon (\nu)=0$ and $C_\nu$ such that
\begin{equation}
\label{kfjmgu453i340y8b90t4my89h0,iwcv9tuh89gur9vfutr78930-34it90tr=0copwq=-pz}
K_{p,n,\upsilon}(R^{-1})\le C_\nu R^{\epsilon(\nu)}\sup_{R^{-1}\le \delta\le 1 }(\delta R)^{\frac{n-1}{2}(\frac12-\frac1p)}C_{p,n,\upsilon}(\delta,\nu)
\end{equation}
for each $R>1$.
\end{theorem}

The rest of this subsection will be devoted to proving this theorem. Before we embark in the proof we explain its role and numerology.

\begin{remark}
\label{addedlaterrem}
First, note that an equivalent  reformulation of Theorem \ref{t1} is
\begin{equation}
\label{e15}
K_{p,n,\upsilon}(\delta)\lesssim_\epsilon \delta^{-\frac{n-1}{2}(\frac12-\frac1p)-\epsilon},
\end{equation}
valid in the subcritical range $2\le p\le \frac{2(n+1)}{n-1}$. This can be seen by interpolating the trivial bound for $p=2$ with the bound at the critical index $p=\frac{2(n+1)}{n-1}.$

The proof of Theorem \ref{t1} in the next subsection will rely on induction on $n$. The case $n=2$ is already known. Theorem \ref{t4} from above will apply unconditionally  in the case $n=3$, and together with the bootstrapping argument will prove Theorem \ref{t1} when $n=3$. Once Theorem \ref{t1} is established in dimension $n-2$ for some $n\ge 4$, equation \eqref{e15} will show that requirement (ii) in Theorem \ref{t4} is satisfied for $\frac{2(n+1)}{n-1}<p<\frac{2(n-1)}{n-3}$. For such a $p$, Theorem \ref{t4} is again applicable in dimension $n$. Thus, when combined with the bootstrapping argument, it will lead to the proof of Theorem \ref{t1} in dimension $n$, by letting $p$ approach the critical index $\frac{2(n+1)}{n-1}$. We note that the increment in the induction step is 2, rather than 1.

Finally, we point out that the presence of the factor $(\delta R)^{\frac{n-1}{2}(\frac12-\frac1p)}$ in \eqref{kfjmgu453i340y8b90t4my89h0,iwcv9tuh89gur9vfutr78930-34it90tr=0copwq=-pz} is rather harmless. To give some intuition on why this is the case, we oversimplify the picture just a little bit and assume that for fixed $p>\frac{2(n+1)}{n-1}$, $\upsilon,\nu$ we have
$$C_{p,n,\upsilon}(\delta,\nu)=C\delta^{-\eta}$$
for some $C,\eta>0$ and each $\delta>0$. We distinguish two cases. First, if $\eta\le \frac{n-1}2(\frac12-\frac1p)$, then we automatically also have $\eta<\frac{n-1}{2}-\frac{n}{p}$. Using these, \eqref{kfjmgu453i340y8b90t4my89h0,iwcv9tuh89gur9vfutr78930-34it90tr=0copwq=-pz} immediately gives
$$K_{p,n,\upsilon}(R^{-1})\lesssim_\nu R^{\epsilon(\nu)}R^{\frac{n-1}{2}-\frac{n}{p}}$$
and Theorem \ref{t1} follows, by letting $\nu\to 0$. In the second case, if $\eta>\frac{n-1}2(\frac12-\frac1p)$, then
$$\sup_{R^{-1}\le \delta\le 1 }(\delta R)^{\frac{n-1}{2}(\frac12-\frac1p)}C_{p,n,\upsilon}(\delta,\nu)\le C_{p,n,\upsilon}(R^{-1},\nu)$$
and the honest equivalence between multilinear and linear decouplings is established. We will carry out the formal argument in the following subsection.
\end{remark}
\bigskip

We now start the proof of Theorem \ref{t4} with  a lemma for paraboloids that are allowed to have one small (possibly zero) principal curvature. The motivation behind this consideration will be explained in the end of the proof of Proposition \ref{hcnyf7yt75ycn8u32r8907n580-9=--qc mvntvu5n8t}.
\begin{lemma}
\label{l2}
Let $n\ge 3$. Fix $\upsilon_1,\ldots,\upsilon_{n-2}\in\{-1,1\}$ and let $|a|\lesssim 1$ be arbitrary, possibly zero. Let $\P_\delta$ be a partition of the neighborhood $\A_\delta$ associated with the hypersurface $H^{n-1}_{(\upsilon_1,\ldots,\upsilon_{n-2},a)}$.

If $\supp(\hat{f})\subset \A_\delta$
then for $p\ge 2$ we have, uniformly over the parameter $|a|\lesssim 1$
$$
\|f\|_p\lesssim \delta^{-\frac12+\frac1p}K_{p,n-1,(\upsilon_1,\ldots,\upsilon_{n-2})}(\delta)(\sum_{\theta\in \P_\delta}\|f_\theta\|_p^p)^{1/p}.
$$
\end{lemma}
Before we embark in the proof of the lemma, we give some heuristics on numerology. A simple $L^p$ orthogonality principle asserts that given a ``suitable'' family of pairwise disjoint subsets $S_1,\ldots,S_{M}$ in $\R^n$ we have
\begin{equation}
\label{e12}
\|f\|_p\lesssim_p M^{1-\frac2p}(\sum_{i=1}^{M} \|f_{S_i}\|_p^p)^{1/p}
\end{equation}
for each $2\le p\le \infty$ and each  $f$ Fourier supported in the union of the $S_i$. We may refer to this as being {\em trivial $l^p$ decoupling}. The word ``suitable'' will refer to the situation when \eqref{e12} can be recovered by interpolating the trivial $L^2$ and $L^\infty$ estimates. One example includes the case when $S_i$ are balls or cubes with pairwise disjoint doubles. We omit the details.

If the sets $S_i$ are not subjected to additional requirements, the universal exponent $1-\frac2p$ of $M$ is sharp. To see this, it suffices to consider the case when $S_i$ are equidistant unit balls with collinear centers.  However, this exponent becomes smaller when geometry is favorable. For example, the $L^p$ decoupling inequality \eqref{e15} corresponds to $M\sim \delta^{-\frac{n-1}{2}}$, and the exponent there is $\frac12-\frac1p$,  half of the universal one. The absence of curvature is an enemy, and one expects a penalty of $\delta^{\frac1{2p}-\frac14}$ for each zero principal curvature.  For example, when $a$ is small (possibly zero),  $H^{n-1}_{(\upsilon_1,\ldots,\upsilon_{n-2},a)}$ has a decoupling constant  $\delta^{\frac1{2p}-\frac14}$ larger than in the case $a\sim 1$.
\bigskip

\begin{proof}
The proof is rather standard, we sketch it briefly. The case $n=3$ is entirely typical, we prefer it only to simplify the notation. We can of course also assume $\upsilon_1=1$. By first performing the trivial decoupling \eqref{e12} in the direction of $e_2$ (which corresponds to the entry $a$), it suffices to prove that for each $\alpha\in [-\frac12,\frac12]$ we have
\begin{equation}
\label{e13}
\|\sum_{\theta\in \P_{\delta,\alpha}}f_\theta\|_p\lesssim K_{p,2,1}(\delta)(\sum_{\theta\in \P_{\delta,\alpha}}\|f_\theta\|_p^p)^{1/p},
\end{equation}
where $\P_{\delta,\alpha}$ consists of those $\theta\in \P_\delta$ that intersect the parabola
$$\{(\xi_1,\alpha,\xi_3):\xi_3=\xi_1^2+a\alpha^2\}.$$

We next show how a standard parabolic change of coordinates will allow us to assume $\alpha=0$.
It is easy to see (the reader is again referred to \cite{BD3}) that \eqref{e13} is equivalent with
$$\int_{B_{\delta^{-1}}}|\int_{[-1/2,1/2]\times [\alpha,\alpha+\delta^{1/2}]}f(\xi_1,\xi_2)e(x_1\xi_1+x_2\xi_2+x_3(\xi_1^2+a\xi_2^2))d\xi_1d\xi_2|^pdx_1dx_2dx_3\lesssim $$$$ K_{p,2,1}(\delta)^p \sum_{I:\delta^{1/2}-\text{interval}}\int_{B_{\delta^{-1}}}|\int_{I\times [\alpha,\alpha+\delta^{1/2}]}f(\xi_1,\xi_2)e(x_1\xi_1+x_2\xi_2+x_3(\xi_1^2+a\xi_2^2))d\xi_1d\xi_2|^pdx_1dx_2dx_3$$
for each $B_{\delta^{-1}}$. It is now rather immediate that we can assume $\alpha=0$, since the image of $B_{\delta^{-1}}$ under the transformation
$$(x_1,x_2,x_3)\mapsto (x_1,x_2+2\alpha a x_3,x_3)$$
is roughly $B_{\delta^{-1}}$.

Note however that all $\theta\in\P_{\delta,0}$ lie in the $\delta$ neighborhood of the cylinder
$$\{(\xi_1,\xi_2,\xi_3):\xi_3=\xi_1^2\},$$
and \eqref{e13} follows immediately from Fubini.
\end{proof}
\bigskip

A simple induction on scales similar to the one in Section 7 in \cite{BD3} allows us to extend the previous lemma to arbitrary hypersurfaces with (at least) $n-2$ principal  curvatures bounded away from zero.

\begin{lemma}
\label{l7}
 Let $n\ge 3$ and $p\ge 2$. Let $S$ be a $C^2$ compact hypersurface in $\R^n$ which at any given point has at least $n-2$ principal curvatures with magnitudes $\sim 1$, while the remaining one is $\lesssim 1$.  Let as usual $\P_\delta$ be a partition of the neighborhood $\A_\delta$ associated $S$.
Assume that for each $\delta'>0$
$$\max_{\upsilon\in\{-1,1\}^{n-2}}K_{p,n-1,\upsilon}(\delta')\lesssim_{\epsilon}{\delta'}^{-\frac{n-2}{2}(\frac12-\frac1p)-\epsilon}.$$

Then for each $\delta$ and each $\supp(\hat{f})\subset \A_\delta$
we have
$$
\|f\|_p\lesssim_{\epsilon}{\delta}^{-\frac{n}{2}(\frac12-\frac1p)-\epsilon}(\sum_{\theta\in \P_\delta}\|f_\theta\|_p^p)^{1/p}.
$$
\end{lemma}
\begin{proof}
For $\delta<1$, let as before $K_{p,n,S}(\delta)$ be the smallest constant such that for each   $f$ with Fourier support in $\A_\delta$ we have
$$\|f\|_p\le K_{p,n,S}(\delta)(\sum_{\theta\in \P_\delta}\|f_\theta\|_p^p)^{1/p}.$$
First note that for each such $f$
$$\|f\|_p\le K_{p,n,S}(\delta^{\frac23})(\sum_{\tau\in \P_{\delta^{\frac23}}}\|f_\tau\|_p^p)^{1/p}.$$
Second, our assumption on the principal curvatures of $S$ combined with Taylor's formula shows that on each $\tau\in \P_{\delta^{\frac23}}$, $S$ is within $\delta$ from a paraboloid $H^{n-1}_{\upsilon}$ with at least $n-2$ of the entries of $\upsilon$ having magnitudes of order 1. By invoking Lemma \ref{l2} for this paraboloid (via a simple rescaling), combined with  parabolic rescaling we get
$$\|f_\tau\|_p\lesssim (\delta^{1/3})^{\frac1p-\frac12}\max_{\upsilon\in\{-1,1\}^{n-2}}K_{p,n-1,\upsilon}(\delta^{1/3})(\sum_{\theta\in \P_\delta:\theta\subset\tau}\|f_\theta\|_p^p)^{1/p}.$$
For each $\epsilon>0$, we conclude the existence of $C_\epsilon$ such that for each $\delta<1$
$$K_{p,n,S}(\delta)\le C_\epsilon [\delta^{-\frac{n}{2}(\frac12-\frac1p)-\epsilon}]^{1/3}K_{p,n,S}(\delta^{\frac23}).$$
By iteration this immediately leads to the desired conclusion.
\end{proof}

We next present a lemma that will play a key role in the proof of Proposition \ref{hcnyf7yt75ycn8u32r8907n580-9=--qc mvntvu5n8t} below.

\begin{lemma}
\label{l:121}
Let $A$ be an invertible symmetric $n\times n$ matrix and let $S$ be an $m$ dimensional affine subspace of $\R^n$. There exists $\delta=\delta(A)$ such that if the $m$ dimensional quadratic hypersurface
$$x_{m+1}=\langle Ax,x\rangle,\;x\in S$$
has $l$ principal curvatures in the interval $[-\delta,\delta]$ then
$$l\le n-m.$$
\end{lemma}
\begin{proof}
We may assume $S$ contains the origin. Choose $\delta$ small enough so that the hypothesis forces the existence of an $l$ dimensional subspace $S_1$ of $S$ such that
\begin{equation}
\label{3dvjher9r576tiocuxmgyutxmcjruit78}
\|P_{S}Ax\|\le \frac12\|A^{-1}\|^{-1}\|x\|
\end{equation}
for each $x\in S_1$. Here $P_S$ is the orthogonal projection onto $S$. We claim that $S_1\cap A^{-1}S=\{0\}$, which will easily imply the desired conclusion. Indeed, otherwise there is $x\in S_1$ with $\|x\|=1$ and $Ax\in S$, and \eqref{3dvjher9r576tiocuxmgyutxmcjruit78} forces the contradiction.
\end{proof}
\bigskip

Here is the basic step in the Bourgain-Guth-type induction on scales that relates the linear and the multilinear decoupling.

\begin{proposition}
\label{hcnyf7yt75ycn8u32r8907n580-9=--qc mvntvu5n8t}
Fix $n\ge 3$, $\upsilon\in\{-1,1\}^{n-1}$ and let $p\ge 2$. Assume one of the following holds

(i) $n=3$

(ii) $n\ge 4$ and $K_{p,n-2,\upsilon'}(\delta')\lesssim_\epsilon {\delta'}^{-\frac{n-3}{2}(\frac12-\frac1p)-\epsilon}$ for each $\delta'>0$, $\upsilon'\in\{-1,1\}^{n-3}$ and  each $\epsilon>0$.

 Then for each $\epsilon$ there exist constants $C_\epsilon$, $C_n$ such that for each $R>1$ and $K\ge 1$
$$\|\widehat{gd\sigma}\|_{L^p(w_{B_R})}\le C_\epsilon[(\sum_{\alpha\subset H^{n-1}_\upsilon\atop{\alpha:\frac{1}{K}-\text{ cap}}}\|\widehat{g_{\alpha}d\sigma}\|_{L^p(w_{B_R)}}^p)^{1/p}+K^{\frac{n-1}{2}(\frac12-\frac1p)+\epsilon}(\sum_{\beta\subset H^{n-1}_\upsilon\atop{\beta:\frac{1}{K^{1/2}}-\text{ cap}}}\|\widehat{g_{\beta}d\sigma}\|_{L^p(w_{B_R})}^p)^{1/p}]+$$$$+K^{C_n}C_{p,n,\upsilon}(R^{-1},K^{-n})(\sum_{\Delta\subset H^{n-1}_\upsilon\atop{\Delta:\frac1{R^{1/2}}-\text{ cap}}}\|\widehat{g_{\Delta}d\sigma}\|_{L^p(w_{B_R})}^p)^{1/p}$$
\end{proposition}
\begin{remark}
The exponent $n$ of $K^{-n}$ in the expression  $C_{p,n,\upsilon}(R^{-1},K^{-n})$ is not important, and not optimal.
\end{remark}
\begin{proof}

We first prove the case $n=3$ and then indicate the modifications needed for $n\ge 4$.

It is rather immediate that if  $Q_1,Q_2,Q_3\subset [-1/2,1/2]^2$, the volume of the parallelepiped spanned by the  unit normals to $H^2_\upsilon$ at $\pi^{-1}(Q_i)$ is comparable to the area of the triangle $\Delta Q_1Q_2Q_3$.

As in \cite{BG}, we may think of $|\widehat{g_{\alpha}d\sigma}|$ as being essentially constant on each ball $B_K$. Denote by $c_\alpha(B_K)$ this value and  let $\alpha^*$ be the cap that maximizes it.

The starting point in the argument is the observation in \cite{BD3} that for each $B_K$ there exists a line $L=L(B_K)$ in the $(\xi_1,\xi_2)$ plane such that if
$$S_L=\{(\xi_1,\xi_2): \dist((\xi_1,\xi_2),L)\le \frac{C}{K^{\frac12}}\}$$
then for $x\in B_K$
$$ |\widehat{gd\sigma}(x)|\le $$
\begin{equation}
\label{term1.1}C\max_{\alpha}|\widehat{g_{\alpha}d\sigma}(x)|+
\end{equation}
\begin{equation}
\label{term1.4addedonJuly15/2015}C\max_{\beta}|\widehat{g_{\beta}d\sigma}(x)|+
\end{equation}
\begin{equation}
\label{term1.2}
K^{4}\max_{\alpha_1,\alpha_2,\alpha_3\atop{K^{-2}-\text{transverse}}}(\prod_{i=1}^3|\widehat{g_{\alpha_i}d\sigma}(x)|)^{1/3}+
\end{equation}\begin{equation}
\label{term1.3}
|\sum_{\beta\subset \pi^{-1}(S_L)\cap H^2_\upsilon}\widehat{g_{\beta}d\sigma}(x)|.
\end{equation}
To see this, we distinguish three scenarios.

First, if $c_\alpha(B_K)\le K^{-2}c_{\alpha^*}(B_K)$ for each $\alpha$ with  $\dist(\pi(\alpha),\pi(\alpha^*))\ge \frac{10}{K^{\frac12}}$, then the sum of \eqref{term1.1} and \eqref{term1.4addedonJuly15/2015} suffices, as
$$ |\widehat{gd\sigma}(x)|\le \sum_{\alpha:\;\dist(\pi(\alpha),\pi(\alpha^*))\ge \frac{10}{K^{\frac12}}}|\widehat{g_\alpha d\sigma}(x)|+|\sum_{\alpha:\;\dist(\pi(\alpha),\pi(\alpha^*))< \frac{10}{K^{\frac12}}}\widehat{g_\alpha d\sigma}(x)|.$$
Otherwise, there is  $\alpha^{**}$ with $\dist(\pi(\alpha^{**}),\pi(\alpha^*))\ge \frac{10}{K^{\frac12}}$ and
$c_{\alpha^{**}}(B_K)\ge K^{-2}c_{\alpha^*}(B_K)$.
The line $L$ is determined by the centers of $\alpha^*,\alpha^{**}$.

Second, if there is $\alpha^{***}$ such that  $\pi(\alpha^{***})$ intersects the complement of $S_L$ and $c_{\alpha^{***}}(B_K)\ge K^{-2}c_{\alpha^*}(B_K)$ then  \eqref{term1.2} suffices. Indeed, note that $\alpha^{*},\alpha^{**}$, $\alpha^{***}$ are $K^{-1}$ transverse by the earlier remark.

The third case is when $c_{\alpha}(B_K)< K^{-2}c_{\alpha^*}(B_K)$ whenever  $\pi(\alpha)$ intersects the complement of $S_L$. It is immediate that the sum of \eqref{term1.1} and \eqref{term1.3} will suffice in this case.

The only  case we need to address is the one corresponding to this latter scenario.  An application of the trivial $l^p$ decoupling \eqref{e12} shows that
$$\|\sum_{\beta:\pi(\beta)\subset S_L}\widehat{g_{\alpha}d\sigma}\|_{L^p(B_K)}\lesssim K^{\frac12-\frac1p}(\sum_{\beta}\|\widehat{g_{\beta}d\sigma}\|_{L^p(w_{B_K})}^p)^{1/p}.$$
This is the best we can say in general. Indeed, in the case of the hyperbolic paraboloid $\upsilon=(1,-1)$, if the line $L$ happens to be $\xi_2=\pm\xi_1$ then $\pi^{-1}(L)$ is itself a line. The absence of curvature prevents any non-trivial estimate to hold.

We conclude that in either case
$$\|\widehat{gd\sigma}\|_{L^p(B_K)}\lesssim [(\sum_{\alpha\subset H^2_\upsilon\atop{\alpha:\frac{1}{K}\text{ cap}}}\|\widehat{g_{\alpha}d\sigma}\|_{L^p(w_{B_K})}^p)^{1/p}+K^{\frac12-\frac1p}(\sum_{\beta\subset H^2_\upsilon\atop{\beta:\frac{1}{K^{1/2}}\text{ cap}}}\|\widehat{g_{\beta}d\sigma}\|_{L^p(w_{B_K})}^p)^{1/p}]+$$$$+K^{10}C_{p,3,\upsilon}(R^{-1},K^{-1})(\sum_{\Delta\subset H^2_\upsilon\atop{\Delta:\frac1{R^{1/2}}\text{ cap}}}\|\widehat{g_{\Delta}d\sigma}\|_{L^p(w_{B_K})}^p)^{1/p}.$$Finally,  raise to the $p^{th}$ power and sum over $B_K\subset B_R$. Also, the norm $\|\widehat{gd\sigma}\|_{L^{p}(B_R)}$ can be replaced by the weighted norm $\|\widehat{gd\sigma}\|_{L^{p}(w_{B_R})}$ via the standard localization argument.

One may repeat this argument in the case $n\ge 4$ as follows. For each $B_K$ there exists a hyperplane $\E=\E(B_K)$ in the $(\xi_1,\ldots,\xi_{n-2})$ space such that
for $x\in B_K$
$$ |\widehat{gd\sigma}(x)|\le $$
$$C\max_{\alpha}|\widehat{g_{\alpha}d\sigma}(x)|+$$
$$C\max_{\beta}|\widehat{g_{\beta}d\sigma}(x)|+$$

$$K^{C_n}\max_{\alpha_1,\ldots,\alpha_n\atop{K^{-n}-\text{transverse}}}(\prod_{i=1}^n|\widehat{g_{\alpha_i}d\sigma}(x)|)^{1/n}+$$
$$|\sum_{\beta\subset \pi^{-1}(S_\E)\cap H^{n-1}_\upsilon}\widehat{g_{\beta}d\sigma}(x)|.$$
Here
$$S_\E=\{(\xi_1,\ldots,\xi_{n-1}): \dist((\xi_1,\ldots,\xi_{n-1}),\E)\lesssim \frac{1}{K^{\frac12}}\}.$$
We only need to explain how to accommodate the previous argument to control the last term. Cover $\pi^{-1}(S_\E)\cap H^{n-1}_\upsilon$ by pairwise disjoint caps $\beta$ of diameter $\sim\frac1{K^{1/2}}$. These caps are inside the $\frac1K$ neighborhood of a cylinder of height $\sim K^{-\frac12}$ over the $n-2$ dimensional manifold $$S_{\E,\upsilon}=\{(\xi_1,\ldots,\xi_n)\in H^{n-1}_\upsilon:(\xi_1,\ldots,\xi_{n-1})\in \E\},$$
and correspond to a tiling of this manifold by $\frac1{K^{1/2}}$- caps. The important new observation is that $S_{\E,\upsilon}$, regarded as an $n-2$ dimensional hypersurface in the hyperplane
$$\{(\xi_1,\ldots,\xi_n):(\xi_1,\ldots,\xi_{n-1})\in \E\}$$
has at least $n-3$ of its $n-2$ principal curvatures bounded away from zero, at any given point. This is of course a consequence of Lemma \ref{l:121}. The case $n=3$ discussed earlier shows that one (in this case the only) principal curvature may indeed happen to be zero. More generally, consider any hyperbolic paraboloid $H^{n-1}_{\upsilon}$. Fix any $A_1,\ldots,A_{n-1}$ such that
$$\sum_{i=1}^{n-1}\upsilon_iA_i^2=0.$$Let $\E$ be the hyperplane
$$\sum_{i=1}^{n-1}\upsilon_iA_i\xi_i=0.$$

It is easy to check that for each point $(\xi_1^*,\ldots,\xi_n^*)$ in the corresponding manifold $S_{\E,\upsilon}$, the (appropriate part of the) line
$$\frac{\xi_1-\xi_1^*}{A_1}=\ldots=\frac{\xi_{n-1}-\xi_{n-1}^*}{A_{n-1}}=\frac{\xi_n-\xi_n^*}{0}$$
is inside $S_{\E,\upsilon}$. In other words $S_{\E,\upsilon}$ is a cylinder, and one of its principal curvatures will be zero.

Using our hypothesis and  Lemma \ref{l7} and Fubini we can write
$$\|\sum_{\beta\subset \pi^{-1}(S_\E)\cap H^{n-1}_\upsilon}\widehat{g_{\beta}d\sigma}\|_{L^p(B_K)}\lesssim_\epsilon K^{\frac{n-1}{2}(\frac12-\frac1p)+\epsilon}(\sum_{\beta}\|\widehat{g_{\beta}d\sigma}\|_{L^p(w_{B_K})}^p)^{1/p}.$$

The argument is now complete.
\end{proof}
\bigskip

Simple parabolic rescaling leads to the following more general version. The interested reader should consult the proof of the analogous result in \cite{BD3} for details.
\begin{proposition}
\label{p9}
Fix $n\ge 3$, $\upsilon\in\{-1,1\}^{n-1}$ and let $p\ge 2$. Assume one of the following holds

(i) $n=3$

(ii) $n\ge 4$ and $K_{p,n-2,\upsilon'}(\delta')\lesssim_\epsilon {\delta'}^{-\frac{n-3}{2}(\frac12-\frac1p)-\epsilon}$ for each $\delta'>0$, $\upsilon'\in\{-1,1\}^{n-3}$ and  each $\epsilon>0$.

 Then for each $\epsilon$ there exist constants $C_\epsilon$, $C_n$ such that for each $R>1$ and $K\ge 1$ and for each $\delta$-cap $\tau$ on $H^{n-1}_\upsilon$ we have
$$\|\widehat{g_\tau d\sigma}\|_{L^p(w_{B_R})}\le C_\epsilon[(\sum_{\alpha\subset \tau\atop{\alpha:\frac{\delta}{K}-\text{ cap}}}\|\widehat{g_{\alpha}d\sigma}\|_{L^p(w_{B_R)}}^p)^{1/p}+K^{\frac{n-1}{2}(\frac12-\frac1p)+\epsilon}(\sum_{\beta\subset \tau\atop{\beta:\frac{\delta}{K^{1/2}}-\text{ cap}}}\|\widehat{g_{\beta}d\sigma}\|_{L^p(w_{B_R})}^p)^{1/p}]+$$$$+K^{C_n}C_{p,n,\upsilon}((R\delta^2)^{-1},K^{-n})(\sum_{\Delta\subset \tau\atop{\Delta:\frac1{R^{1/2}}-\text{ cap}}}\|\widehat{g_{\Delta}d\sigma}\|_{L^p(w_{B_R})}^p)^{1/p}.$$
\end{proposition}

We are now ready to prove Theorem \ref{t4}. Let $K=\nu^{-1/n}$.
Iterate Proposition \ref{p9} starting with caps of scale $ 1$ until all resulting caps have scale  $R^{-1/2}$.
Each iteration  lowers the scale of the caps from $\delta$ to at least $\frac{\delta}{K^{1/2}}$. When iteration is over, we end up with a sum of terms of the form
$$T_\Gamma=\Gamma K^{C_n}(\sum_{\Delta:\frac1{R^{1/2}}-\text{ cap}}\|\widehat{g_{\Delta}d\sigma}\|_{L^p(w_{B_R})}^p)^{1/p},$$
with various coefficients $\Gamma$.
Each such term  arises via $\le\log_K R$ iterations. Also, a crude estimate shows that we end up with at most $3^{\log_K R}=R^{O(\log_{\nu^{-1}}3)}$ such terms.

It remains to get a uniform upper bound on $\Gamma$. Tracing back the iteration history of $T_\Gamma$, assume it went through $m_1$ steps where scale was lowered by $K$ and $m_2$ steps where scale was lowered by $K^{1/2}$. Then obviously, for each $\epsilon$
$$\Gamma\le (C_{\epsilon})^{m_1+m_2}K^{[\frac{n-1}{2}(\frac12-\frac1p)+\epsilon]m_2}C_{p,n,\upsilon}((RK^{-m_2-2m_1})^{-1},\nu).$$
Using  the bound $m_1+m_2\le \log_K R$ this is further bounded by
$$R^{\log_{\nu^{-1}}C_\epsilon}K^{\epsilon\log_KR}C_{p,n,\upsilon}((RK^{-m_2-2m_1})^{-1},\nu)(RK^{-m_2-2m_1})^{-\frac{n-1}{2}(\frac12-\frac1p)}R^{\frac{n-1}{2}(\frac12-\frac1p)}\le $$
$$\le R^{\log_{\nu^{-1}}C_\epsilon}K^{\epsilon\log_KR}\sup_{R^{-1}\le \delta\le 1 }(\delta R)^{\frac{n-1}{2}(\frac12-\frac1p)}C_{p,n,\upsilon}(\delta,\nu).$$
The proof is now complete, by carefully letting $\epsilon$ approach zero at slower rate than $\nu$.

\bigskip

\subsection{The bootstrapping argument}

We now enter the second and final stage of the argument for Theorem \ref{t1}. We recommend the reader to check Remark \ref{addedlaterrem} for  a high level overview of the argument.   For the remainder of the section we fix $\upsilon\in\{-1,1\}^{n-1}$ and $p>\frac{2(n+1)}{n-1}$.
To simplify notation we let $K(\delta)=\delta^{\frac{n-1}{2}(\frac12-\frac1p)}K_{p,n,\upsilon}(\delta)$.

Let $\gamma$ be the unique number such that
$$\lim_{\delta\to 0}K(\delta)\delta^{\gamma+\epsilon}=0,\;\;\text{for each }\epsilon>0$$
and
\begin{equation}
\label{o54mitnvpcmi2eoptunv890cixw0qygfd390}
\limsup_{\delta\to 0}K(\delta)\delta^{\gamma-\epsilon}=\infty,\;\;\text{for each }\epsilon>0.
\end{equation}
Write $\gamma=\frac{n-1}{4}-\frac{n+1}{2p}+\alpha$. Recall that we need to prove that $\alpha=0$.

Define
\begin{equation}
\label{jjjeryyyrdewe678wer6wr5drtftdujhdji}
\xi=\frac{2}{(p-2)(n-1)},\;\;\;\;\eta=\frac{n(np-2n-p-2)}{2p(n-1)^2(p-2)}.
\end{equation}
Since $p>\frac{2(n+1)}{n-1}$ we have that $\xi<\frac12$. A simple computation reveals that the assumption $\alpha>0$ is equivalent with
$$\gamma\frac{1-\xi}{1-2\xi}>\frac{n-1}{4}-\frac{n^2+n}{2p(n-1)}+\frac{2\eta}{1-2\xi}.$$
Under this assumption it follows that we can choose $s_0\in\N$  large enough and  $\nu>0$ small enough such that, with  $\epsilon(\nu)$ as in  Theorem \ref{t4}, we have
$$\gamma(\frac{1-\xi}{1-2\xi}-\frac{\xi(2\xi)^{s_0}}{1-2\xi})>
$$
\begin{equation}
\label{mamatata}
>\frac{n-1}{4}-\frac{n^2+n}{2p(n-1)}+2^{s_0}\epsilon(\nu)+\frac{2\eta}{1-2\xi}(1-(2\xi)^{s_0})+\frac{n}{(n-1)p}(2\xi)^{s_0}.
\end{equation}

Note that $s_0$ and $\nu$ depend only on the fixed parameters $p,n,\alpha$. As a result, we follow our convention and do not record the dependence on them when using the symbol $\lesssim$.

Throughout the rest of the section $\nu$ and $s_0$  will always refer to these values.
Introduce the following semi-norms
$$\|f\|_{p,\delta,B}=(\sum_{\theta\in \P_\delta}\|f_\theta\|_{L^p(w_B)}^2)^{1/2},$$
$$|||f|||_{p,\delta,B}=\delta^{-\frac{n-1}{2}(\frac12-\frac1p)}(\sum_{\theta\in \P_\delta}\|f_\theta\|_{L^p(w_B)}^p)^{1/p}$$
and note that
\begin{equation}
\label{e22}
\|f\|_{p,\delta,B}\le |||f|||_{p,\delta,B}.
\end{equation}

For a fixed $0\le \beta\le 1$ consider the inequality
\begin{equation}
\label{inv999}
\|(\prod_{i=1}^n|\widehat{g_id\sigma}|)^{1/n}\|_{L^p(B_N)}
\lesssim_{\epsilon} A_\beta(N)N^{\epsilon}X(B_N)^{1-\beta}Y(B_N)^\beta,
\end{equation}
for arbitrary $\epsilon>0$, $N$, $g_i$ and $B_N$ as before.
Here
$$X(B_N)=(\prod_{i=1}^n|||\widehat{g_id\sigma}|||
_{p,\delta,B_N})^{\frac1{n}},$$
$$Y(B_N)=(\prod_{i=1}^n|||\widehat{g_id\sigma}|||
_{\frac{p(n-1)}{n},\delta,B_N})^{\frac1{n}}.$$

The following holds.
\begin{proposition}
\label{propinv5}
(a) Inequality \eqref{inv999} holds true for $\beta=1$ with $A_1(N)=N^{\frac{n-1}{4}-\frac{n^2+n}{2p(n-1)}}$.

(b) Moreover, if we assume \eqref{inv999} for some $\beta\in (0,1]$, then we also have \eqref{inv999} for $\frac{2\beta}{(p-2)(n-1)}$ with
$$A_{\frac{2\beta}{(p-2)(n-1)}}(N)= A_\beta(N^{1/2})\delta^{-\frac{\gamma}{2}(1-\frac{2\beta}{(p-2)(n-1)})}N^{\frac{n(np-2n-p-2)}{2p(n-1)^2(p-2)}\beta}.$$
\end{proposition}

The proof follows line by line the proof of the analogous Proposition 6.3 in \cite{BD3}. More precisely, part (a) here follows right away from Proposition 6.3 (i) by simply invoking \eqref{e22}. Also, the only modification needed to prove part (b) is to notice that the following consequence of H\"older's inequality used in \cite{BD3}
$$\|\widehat{g_id\sigma}\|_{\frac{p(n-1)}{n},\delta,B_N}\le \|\widehat{g_id\sigma}\|_{p,\delta,B_N}^{1-\frac{2}{(p-2)(n-1)}}\|\widehat{g_id\sigma}\|_{2,\delta,B_N}^{\frac{2}{(p-2)(n-1)}}$$
continues to hold if $\|\cdot\|$ is replaced with $|||\cdot|||$.

Proposition \ref{propinv5} implies that for each $s\ge 0$
$$A_{\xi^s}(N)= N^{\psi(\xi^s)}$$
with
\begin{equation}
\label{inv13}
\psi(\xi^{s+1})= \frac12\psi(\xi^s)+\frac\gamma2(1-\xi^{s+1})+\eta\xi^s.
\end{equation}
Recall that $\xi<\frac12$. Iterating \eqref{inv13}  gives
\begin{equation}
\label{inv15}
\psi(\xi^s)=\frac1{2^s}\psi(1)+\gamma(1-2^{-s})+2(\frac\eta\xi-\frac\gamma2)\frac{2^{-s}-\xi^{s}}{\xi^{-1}-2}
\end{equation}

Note that $Y(B_N)\lesssim X(B_N)N^{\frac{n}{(n-1)p}}$. As \eqref{inv999} holds for $\beta=\xi^s$ and arbitrary $\nu$-transverse caps $\tau_i$ we get
\begin{equation}
\label{7743048765898}
\delta^{\frac{n-1}{2}(\frac12-\frac1p)}C_{p,n,\upsilon}(\delta,\nu)\lesssim_{\epsilon,s} \delta^{-\epsilon}A_{\xi^s}(N)N^{\frac{n\xi^s}{(n-1)p}}, \;\;\text{for each }\epsilon>0.
\end{equation}
\bigskip

To finish the proof of Theorem \ref{t1} for $H^{n-1}_\upsilon$, we will argue using induction on $n$ that $\alpha=0$. As observed earlier, the case $n=2$ is covered by the main theorem in \cite{BD3}. So the first case to consider is $n=3$. Using Theorem \ref{t4} and \eqref{7743048765898} we get
\begin{equation}
\label{7743048765898newww}
K(\delta)\lesssim_{\epsilon,s}\delta^{-\epsilon-\epsilon(\nu)}\sup_{1\le M\le N}A_{\xi^s}(M)M^{\frac{n\xi^s}{(n-1)p}}=\delta^{-\epsilon-\epsilon(\nu)}N^{\psi(\xi^{s})+\frac{n\xi^{s}}{(n-1)p}}.
\end{equation}
Since \eqref{7743048765898newww} (with $s=s_0$) holds for arbitrarily small $\delta$ and $\epsilon$, we further get by invoking \eqref{o54mitnvpcmi2eoptunv890cixw0qygfd390} that
\begin{equation}
\label{inv14}
\gamma\le \psi(\xi^{s_0})+\frac{n\xi^{s_0}}{(n-1)p}+\epsilon(\nu).
\end{equation}
Combining \eqref{inv15} and \eqref{inv14}  we find
$$\gamma(\frac{1-\xi}{1-2\xi}-\frac{\xi(2\xi)^{s_0}}{1-2\xi})\le \psi(1)+2^{s_0}\epsilon(\nu)+\frac{2\eta}{1-2\xi}(1-(2\xi)^{s_0})+\frac{n}{(n-1)p}(2\xi)^{s_0},$$
which  contradicts \eqref{mamatata}. Thus $\alpha=0$ and Theorem \ref{t1} is proved for $n=3$ and $p>4$.

Assume now that $n\ge 4$ and that Theorem \ref{t1} was proved in dimension  $n-2$. To prove Theorem \ref{t1} in $\R^n$ for $p>\frac{2(n+1)}{n-1}$, it suffices to prove it for $\frac{2(n+1)}{n-1}<p<\frac{2(n-1)}{n-3}$. Note that in this range we have $p<\frac{2(d+1)}{d-1}$ for $d=n-2$, in particular \eqref{e42} holds. Thus Theorem \ref{t4} is applicable due to our induction hypothesis and we reach a contradiction as in the case $n=3$ discussed above.
\bigskip

It remains to see why Theorem \ref{t1} holds for the endpoint $p=p_n=\frac{2(n+1)}{n-1}$. Via a localization argument, $K_{p,n,\upsilon}(\delta)$ is comparable to the best constant $K_{p,n,\upsilon}^{*}(\delta)$ that makes the following inequality true for each $N$-ball $B_N$ and each $f$ Fourier supported in $\A_\delta$
\begin{equation}
\label{inv50}
\|f\|_{L^p({B_N})}\le K_{p,n,\upsilon}^{*}(\delta)(\sum_{\theta\in \P_\delta}\|f_\theta\|_{L^p(\R^n)}^p)^{1/p}.
\end{equation}
It suffices now to invoke Theorem \ref{t1} for $p>\frac{2(n+1)}{n-1}$ together with
$$\|f\|_{L^{p_n}({B_N})}\lesssim \|f\|_{L^{p}({B_N})}N^{\frac{n}{p_n}-\frac{n}{p}}\;\;\;\text{(by H\"older's inequality)},$$
$$\|f_\theta\|_{L^p(\R^n)}\lesssim N^{\frac{n+1}{2p}-\frac{n+1}{2p_n}}\|f_\theta\|_{L^{p_n}(\R^n)}\;\;\;\text{(by Bernstein's inequality)},$$
and then to let $p\to p_n$.

\bigskip

\section{The proof of Theorem \ref{t1:l2notlp}}

\label{sec:Str}

We start this section by explaining  the numerology in Theorem \ref{t1:l2notlp}, in particular the origin of the critical index $\frac{2(n+1-d(\upsilon))}{n-1-d(\upsilon)}$.
There are two examples to consider.

Example 1 relates to the fact that the essentially sharp estimate
$$|\widehat{d\sigma_{H^{n-1}_\upsilon}}(x)|\lesssim (1+|x|)^{\frac{1-n}{2}}$$
is universal, it does not depend on the signature of $H^{n-1}_\upsilon$. Using $f$ such that $\widehat{f}$ is (a smooth approximation of) the characteristic function of the $\delta$ neighborhood of $H^{n-1}_\upsilon$,  leads to the lower bound
\begin{equation}
\label{eweinn1}
K_{n,p,\upsilon}^{(2)}(\delta)\gtrsim \delta^{-\frac{n-1}4+\frac{n+1}{2p}},\;\;\;p\ge\frac{2(n+1)}{n-1},
\end{equation}
which was shown to be sharp in \cite{BD3} in the elliptic case.

In the non elliptic case we have that $d(\upsilon)\ge 1$ and $H^{n-1}_\upsilon$ will contain a compact subset $V$ of an affine subspace of dimension $d(\upsilon)$.

Example 2 is concerned with the case when $\widehat{f}$ is the characteristic function of  the $\delta$ neighborhood of $V$. A standard computation shows that for  this $f$ we have
$$\|f\|_p\sim \delta^{d(\upsilon)(\frac1{2p}-\frac14)}(\sum_{\theta}\|f_\theta\|_p^2)^{1/2},\;\;\;p\ge 2$$
which leads to
\begin{equation}
\label{eweinn2}
K_{n,p,\upsilon}^{(2)}(\delta)\gtrsim \delta^{d(\upsilon)(-\frac14+\frac1{2p})},\;\;\;p\ge 2.
\end{equation}
Now \eqref{eweinn1} and \eqref{eweinn2} suggest that
$$K_{n,p,\upsilon}^{(2)}(\delta)\sim \max\{\delta^{-\frac{n-1}4+\frac{n+1}{2p}},\delta^{d(\upsilon)(-\frac14+\frac1{2p})}\},\;\;\;p\ge 2.$$
We will prove that this is indeed correct. Note that there is a regime change precisely at $p=\frac{2(n+1-d(\upsilon))}{n-1-d(\upsilon)}$.

It is worth mentioning that Example 1 and 2 above also apply (with the correct change in numerology) to the case of $l^p$ decouplings in Theorem \ref{t1}. However,  one may check that when $p\ge 2\frac{n+1}{n-1}$ the lower bound coming from Example 1 always dominates the one produced by Example 2.

\bigskip

The proof of Theorem \ref{t1:l2notlp} follows via induction on the dimension $n$. The case $n=2$ was proved in \cite{BD3}. Assume we have proved the theorem for all $H_{\upsilon}^{n-2}$ with $\upsilon\in(\R\setminus \{0\})^{n-2}$, for some fixed $n\ge 3$.
Fix now $\upsilon\in (\R\setminus \{0\})^{n-1}$. By invoking interpolation with $L^2$ and $L^\infty$, it suffices to prove the theorem for $H_{\upsilon}^{n-1}$ when $p$ is greater than but arbitrarily close to the critical index $\frac{2(n+1-d(\upsilon))}{n-1-d(\upsilon)}$. In particular, it suffices to consider $\frac{2(n+1-d(\upsilon))}{n-1-d(\upsilon)}< p<\frac{2(n-d(\upsilon))}{n-2-d(\upsilon)}$.

\bigskip

The guiding principle is that whenever curvature is absent one uses the trivial $l^2$ decoupling for $2\le p<\infty$, which amounts to the following. Given any pairwise disjoint parallelepipeds  $S_1,\ldots,S_{M}$ in $\R^n$ we have

\begin{equation}
\label{e122}
\|f\|_p\le M^{\frac12-\frac1p}(\sum_{i=1}^{M} \|f_{S_i}\|_p^2)^{1/2}
\end{equation}
for each $2\le p\le \infty$ and each  $f$ Fourier supported in the union of the $S_i$. We will refer to this as being {\em trivial $l^2$ decoupling}. Note that this is an analogue of \eqref{e12}, which follows again by interpolation.

\bigskip

We now present the main steps in the argument.
First, our induction hypothesis implies the following result for hypersurfaces with  small (possibly zero) principal curvatures, in the style of Lemma \ref{l7}. More precisely, for a  given  hypersurface $S$ in $\R^{n-1}$ ($n\ge 3$) with $p(S)$ principal curvatures $\ge 1$, $q(S)$ principal curvatures $\le -1$ and $r(S)$ principal curvatures in $(-1,1)$,  define  $$d(S)=r(S)+\min(p(S),q(S)).$$
\begin{lemma}
\label{l7dyyfebtt7rw56w6xnt}
Assume Theorem \ref{t1:l2notlp} holds in $n-1$ dimensions.
If $\supp(\hat{f})\subset \A_\delta(S)$ and $2\le p\le \frac{2(n-d(S))}{n-2-d(S)}$ we have
\begin{equation}
\label{e6l2notlpi tvn7ytubfhgnldk'al]'l.ofj}
\|f\|_p\lesssim_\epsilon \delta^{-\epsilon}K^{(2)}_{S,p}(\delta)(\sum_{\theta\in \P_\delta}\|f_\theta\|_p^2)^{1/2},
\end{equation}
where $K^{(2)}_{S,p}(\delta)=\delta^{d(S)(-\frac14+\frac1{2p})}$.
\end{lemma}
\begin{proof}

Of course, there is nothing special about 1 and -1. This result is about ``large" versus ``small" principal curvatures. As before, we may assume $S=H_{\upsilon'}^{n-2}$ where $\upsilon'$ has $p(S)$ entries equal to $1$,  $q(S)$ entries equal to $-1$ and $r(S)$ entries in $(-1,1)$. The case $r(S)=0$ is an immediate consequence of our induction hypothesis. If $r(S)\ge 1$, perform first a trivial $l^2$ decoupling in the direction of each of the $r(S)$ ``small" principal curvatures. This will contribute a factor of $\delta^{r(S)(-\frac14+\frac1{2p})}$ to $K^{(2)}_{S,p}(\delta)$, according to \eqref{e122}. Then use the induction hypothesis for the cross sections $S'$ corresponding to the remaining $p(S)+q(S)$ curvatures. Note that $d(S')=\min(p(S),q(S))$ and that $S'$ is a hypersurface with nonzero Gaussian curvature in $\R^{n-1-r(S)}$. As a result, the critical index for $S'$ is $\frac{2(n-d(S))}{n-2-d(S)}$, and the induction hypothesis is indeed applicable. Thus the cross sections contribute $\delta^{d(S')(-\frac14+\frac1{2p})}$ to $K^{(2)}_{S,p}(\delta)$. Finally, the two contributions to $K^{(2)}_{S,p}(\delta)$ can be pieced together by invoking Fubini, as explained in the proof of Lemma \ref{l2}.
\end{proof}
Next, we record the following analogue of both Proposition \ref{hcnyf7yt75ycn8u32r8907n580-9=--qc mvntvu5n8t} from  here and Proposition 5.5 from \cite{BD3}.
\begin{proposition}
\label{hcnyf7yt75ycn8u32r8907n580-9=--qc mvntvu5n8tnewwww}Assume Theorem \ref{t1:l2notlp} holds in  dimensions less than or equal to $n-1$. Let $2\le p\le\frac{2(n-d(\upsilon))}{n-2-d(\upsilon)}$.
For each $\epsilon$ there exist constants $C_\epsilon$, $C_n$ such that for each $R>1$ and $K\ge 1$
$$\|\widehat{gd\sigma}\|_{L^p(w_{B_R})}\le C_\epsilon[(\sum_{\alpha\subset H^{n-1}_\upsilon\atop{\alpha:\frac{1}{K}-\text{ cap}}}\|\widehat{g_{\alpha}d\sigma}\|_{L^p(w_{B_R)}}^2)^{1/2}+K^{d(\upsilon)(\frac1{2p}-\frac14)+\epsilon}(\sum_{\beta\subset H^{n-1}_\upsilon\atop{\beta:\frac{1}{K^{1/2}}-\text{ cap}}}\|\widehat{g_{\beta}d\sigma}\|_{L^p(w_{B_R})}^2)^{1/2}]+$$$$+K^{C_n}C_{p,n,\upsilon}^{(2)}(R^{-1},K^{-n})(\sum_{\Delta\subset H^{n-1}_\upsilon\atop{\Delta:\frac1{R^{1/2}}-\text{ cap}}}\|\widehat{g_{\Delta}d\sigma}\|_{L^p(w_{B_R})}^2)^{1/2}.$$
Here, in analogy with earlier notation, $C_{p,n,\upsilon}^{(2)}$ is the multilinear version of $K_{p,n,\upsilon}^{(2)}$.
\end{proposition}
\begin{proof}
This follows as before by using the Bourgain--Guth induction on scales, and our induction hypothesis for arbitrary cross sections $S'$ of $H^{n-1}_\upsilon$ with a "vertical" hyperplane. It suffices to prove that $d(S')\le d(\upsilon)$. This follows via an argument similar to the one in Lemma \ref{l:121}, we sketch it briefly.
Assume for simplicity that $\upsilon_i=1$ for $1\le i\le  p(\upsilon)$ and $\upsilon_i=-1$ for $p(\upsilon)+1\le i\le n-1$. Assume also  that the hyperplane contains the origin, and let $H'$ be its intersection with the hyperplane $x_n=0$.  Call $A:\R^{n-1}\to\R^{n-1}$ the linear transformation so that $H_{\upsilon}^{n-1}$ is (part of) the graph of $x\mapsto \langle Ax, x\rangle$, $x\in\R^{n-1}$. Obviously $\upsilon_i$ are its eigenvalues and $\langle e_i \rangle$ are the corresponding eigenspaces.
Define
$$X_+=\langle e_i:\;1\le i\le p(\upsilon)\rangle,\;\;\;\;X_-=\langle e_i:\;p(\upsilon)+1\le i\le n-1\rangle.$$
Similarly, let $A':H'\to H'$ be the linear transformation so  that $S'$ is (part of) the graph of $x\mapsto\langle A'x,x\rangle$ over $x\in H'$.  Let $e_1',\ldots,e_{n-2}'$ be an appropriate orthonormal basis of $H'$ consisting of  eigenvectors of $A'$ with eigenvalues $\upsilon_i'$. Assume that $\upsilon_i'\ge 1$ for $1\le i\le  p(S')$, $\upsilon_i'\le -1$ for $p(S')+1\le i\le p(S')+q(S')$ and $|\upsilon_i'|< 1$ for $p(S')+q(S')+1\le i\le p(S')+q(S')+r(S')=n-2$.
Define
$$X_+'=\langle e_i':\;1\le i\le p(S')\rangle,\;\;\;\;X_-'=\langle e_i':\;p(S')+1\le i\le p(S')+q(S')\rangle,$$$$X_0'=\langle e_i':\;p(S')+q(S')+1\le i\le n-2\rangle.$$
It suffices now to prove that $p(S')+r(S')\le p(\upsilon)$ and $q(S')+r(S')\le q(\upsilon)$. Assume  for contradiction that $p(S')+r(S')> p(\upsilon)$. Then $(X_+'\oplus X_0')\cap X_-$ must contain a unit vector $x$. Note first that since $x\in H'$, we have $\langle Ax,x\rangle=\langle A'x,x\rangle$. On the other hand, since $x\in  X_-$ we must have $\langle Ax,x\rangle = -1$, while $x\in X_+'\oplus X_0'$ implies that $\langle Ax,x\rangle > -1$. The contradiction is now obvious.
\end{proof}

Iterations of Proposition \ref{hcnyf7yt75ycn8u32r8907n580-9=--qc mvntvu5n8tnewwww} show the equivalence between $C_{p,n,\upsilon}^{(2)}$ and $K_{p,n,\upsilon}^{(2)}$. Since we work with $p>\frac{2(n+1-d(\upsilon))}{n-1-d(\upsilon)}$, we also have that $p>\frac{2(n+1)}{n-1}$. In particular, the parameter $\xi$ from \eqref{jjjeryyyrdewe678wer6wr5drtftdujhdji} is $<\frac12$, as desired. The rest of the argument continues exactly as in the elliptic case from \cite{BD3}, see also the  $l^p$ decouplings from the previous section. This is due to the fact that when $p>\frac{2(n+1-d(\upsilon))}{n-1-d(\upsilon)}$, \eqref{e41} and \eqref{e6l2notlp} are identical.
\bigskip

\section{Multilinear versus linear decoupling for curves}
In this section we start the proof of  Theorem \ref{t2}. It is easy to see that for each $t_0\in [0,1]$, there is an affine transformation $L_{t_0}$ of $\R^n$, more precisely a rotation followed by a translation, such that
$$\begin{cases}L_{t_0}(\Phi(t_0))=\textbf{0} \\ L_{t_0}(\Phi'(t_0))\perp\langle e_2,\ldots,e_n\rangle\\ \hfill L_{t_0}(\Phi''(t_0))\perp\langle e_3,\ldots,e_n\rangle \\  \ldots\ldots\ldots\ldots\ldots\ldots\\  L_{t_0}(\Phi^{(n-1)}(t_0))\perp\langle e_n\rangle \end{cases}$$
In this new local system of coordinates, the equation of the curve near $t=0$ becomes
$$\tilde{\Phi}(t)=(C_{1,1}t+C_{1,2}t^2+\ldots+C_{1,n}t^n,C_{2,2}t^2+\ldots+C_{2,n}t^n,\ldots,C_{n,n}t^n)+O(t^{n+1},t^{n+1},\ldots,t^{n+1}).$$
The coefficients $C_{i,j}$  depend on $t_0$ but satisfy $\kappa\le |C_{i,i}|\le \kappa^{-1}$ and $|C_{i,j}|\le \kappa^{-1}$ for $i<j$, with $\kappa>0$ independent of $t_0$, due to the Wronskian condition.

By invoking a simple induction on scales argument, it suffices to prove Theorem \ref{t2} for the special curves $\Phi_{\textbf{C}}$, $\textbf{C}=(C_{i,j})_{1\le i\le j\le n}$
\begin{equation}
\label{e1}
\Phi_{\textbf{C}}(t)=(C_{1,1}t+C_{1,2}t^2+\ldots+C_{1,n}t^n,C_{2,2}t^2+\ldots+C_{2,n}t^n,\ldots,C_{n,n}t^n),
\end{equation}
with $C_{i,j}$ as above.  The estimates will of course depend only on $\kappa$. To see this, for $\delta<1$ and $p\le 4n-2$, let $K_p^*(\delta)$ be the smallest constant such that for each   $f$ with Fourier support in the neighborhood $\A_\delta$ of $\Phi$ we have
$$\|f\|_{p}\le K_p^*(\delta)(\sum_{\theta\in \P_\delta}\|f_\theta\|_{p}^2)^{\frac12}.$$

First, for each such  $f$
\begin{equation}
\label{e8}
\|f\|_{p}\le K_p^*(\delta^{\frac{n}{n+1}})(\sum_{\tau\in \P_{\delta^{\frac{n}{n+1}}}}\|f_\tau\|_{p}^{2})^{\frac12}.
\end{equation}
The previous discussion shows that  the portion of $\Phi$ inside a given $\tau\in \P_{\delta^{\frac{n}{n+1}}}$ is within $\delta$ from a curve \eqref{e1}. Let $a$ be  the left endpoint of the interval of length $\delta^{\frac1{n+1}}$ corresponding to $\tau$. We will perform  rescaling as follows.
Consider the linear transformation
\begin{equation}
\label{fuyvt6rvrgyuyioeynbf7vynciuyyf}
L_\tau(\xi_1,\ldots,\xi_n)=(\xi_1',\ldots,\xi_n')=(\frac{\xi_1-a}{\delta^{\frac1{n+1}}},\frac{\xi_2-2a\xi_1+a^2}{\delta^{\frac2{n+1}}},\frac{\xi_3-3a\xi_2+3a^2\xi_1-a^3}{\delta^{\frac3{n+1}}},\ldots).
\end{equation}

It maps $\tau\cap \A_\delta$ into $\A_{\delta^{\frac{1}{n+1}}}$ and each $\theta\in\A_\delta$ with $\theta\subset \tau$ into some $\theta'\in\A_{\delta^{\frac{1}{n+1}}}$. Using this change of variables and Theorem \ref{t2} with $\delta$ replaced with $\delta^{\frac1{n+1}}$ we get
\begin{equation}
\label{e7}
\|f_\tau\|_{p}\lesssim_\epsilon\delta^{-\epsilon}(\sum_{\theta\in \P_\delta:\theta\subset\tau}\|f_\theta\|_{p}^2)^{\frac12}.
\end{equation}
For each $\epsilon>0$, using \eqref{e8} and \eqref{e7} we conclude the existence of $C_\epsilon$ such that for each $\delta<1$
$$K_p^*(\delta)\le C_\epsilon\delta^{-\epsilon}K_p^*(\delta^{\frac{n}{n+1}}).$$
By iteration this immediately leads to $K_p^*(\delta)\lesssim_\epsilon \delta^{-\epsilon}$.

We further observe that it suffices to consider curves $\Phi_{\textbf{C}}$ with $\textbf{C}$ equal to the identity matrix. This is because the decoupling inequality is preserved under linear transformations. For the remainder of the section we let
$$\Phi(t)=(t,t^2,\ldots,t^{n}).$$

For each dyadic interval $I\subset [0,1]$ define the extension operator
$$E_Ig(x)=\int_Ig(t)e(tx_1+t^2x_2+\ldots+t^nx_n)dt.$$
Unless specified otherwise, all intervals will be implicitly assumed to be dyadic intervals in $[0,1]$.

We will denote by $K_p(R)$ the smallest constant such that
$$\|E_{[0,1]}g\|_{L^{p}(B_R)}\le K_p(R)(\sum_{|U|=R^{-1/n}}\|E_{U}g\|_{L^{p}(w_{B_R})}^2)^{\frac{1}{2}}$$
holds true for each $g:[0,1]\to \C$ and $B_R$.

It is easy to see that $K_p(R)\sim K_p^*(R^{-1})$. In particular, Theorem \ref{t2} for some $p$ is in fact equivalent with the following inequality
\begin{equation}
\label{e11}
\|E_{[0,1]}g\|_{L^{p}(B_N)}\lesssim_{\epsilon}N^{\epsilon}(\sum_{|U|=N^{-1/n}}\|E_{U}g\|_{L^{p}(w_{B_N})}^2)^{\frac{1}{2}}.
\end{equation}

For dyadic $0<\nu<1$, denote  by $C_p(R,\nu)$ the smallest constant such that
$$\|(\prod_{i=1}^n|E_{I_i}g|)^{1/n}\|_{L^{p}(B_R)}\le C_p(R,\nu) (\prod_{i=1}^n\sum_{U\subset I_i\atop{|U|=R^{-1/n}}}\|E_{U}g\|_{L^{p}(w_{B_R})}^2)^{\frac{1}{2n}}$$
holds for each  nonadjacent dyadic intervals $I_1,\ldots,I_n\subset [0,1]$ of size $\nu$, each  $g:[0,1]\to \C$ and $B_R$. We will refer to such intervals as being $\nu$-{\em transverse}.

It is immediate that $C_p(R,\nu)\le K_p(R)$. We now show that the reverse inequality is also essentially true.
\begin{theorem}
\label{wt1}
For each $\nu$ there exists $\epsilon(\nu)$ with $\lim_{\nu\to 0}\epsilon(\nu)=0$ and $C_\nu$ such that for each $R$
$$K_p(R)\le C_{\nu} R^{\epsilon(\nu)}C_p(R,\nu).$$
\end{theorem}

The proof of the theorem is a simple version of the Bourgain--Guth induction on scales \cite{BG}. We will prove a few preliminary results.
\begin{proposition}
\label{p3}
For each $R$-ball $B_R$ we have
$$\|E_{[0,1]}g\|_{L^{p}(B_R)}\le C_{p}(R,\nu)C_\nu(\sum_{|U|=R^{-1/n}}\|E_{U}g\|_{L^{p}(w_{B_R})}^{2})^{1/2}+$$
$$+C(\sum_{|I|=\nu}\|E_{I}g\|_{L^p(B_R)}^{2})^{1/2}.$$
The constant $C$ is independent of $\nu$, it only depends on $n$ and $p$.
\end{proposition}
\begin{proof}
We start by writing
$$E_{[0,1]}g=\sum_{I:\nu-\text{interval }}E_Ig.$$
 It is rather immediate that for each $x\in B_R$
$$|E_{[0,1]}g(x)|\le C\max_{I:\nu-\text{interval}}|E_Ig(x)|+C_\nu\sum_{I_1,\ldots,I_{n}:\;\nu-\text{transverse}}(\prod_{j=1}^{n}|E_{I_j}g(x)|)^{\frac1{n}}\le$$
$$C(\sum_{I:\nu-\text{interval}}|E_Ig(x)|^{2})^{1/2}+C_\nu\sum_{I_1,\ldots,I_{n}:\;\nu-\text{transverse}}(\prod_{j=1}^{n}|E_{I_j}g(x)|)^{\frac1{n}}.$$
The result now follows by integrating the $p$-th power.
\end{proof}
We now rescale to get the following version.
\begin{proposition}
\label{p4}
Let $H$ be a $\delta$- interval in $[0,1]$.
For each $R\gtrsim \delta^{-n}$ and each $R$-ball $B_R$ we have
$$\|E_{H}g\|_{L^{p}(B_R)}\le C_{\nu}C_p(R\delta^n,\nu)(\sum_{|U|=R^{-1/n}}\|E_{U}g\|_{L^{p}(w_{B_R})}^{2})^{1/2}+$$
$$+C(\sum_{J:\delta\nu-\text{interval}\atop{J\subset H}}\|E_{J}g\|_{L^{p}(B_R)}^{2})^{1/2}.$$
\end{proposition}
\begin{proof}
Note that if $H=[a,a+\delta]$ then the change of variables $s=\frac{t-a}{\delta}$ shows that
$$|E_Hg(x)|=\delta|E_{[0,1]}g^{a,\delta}(x')|$$
where $g^{a,\delta}(s)=g(s\delta+a)$ and $x'=(x_1',\ldots,x_n')$ with
$$\begin{cases}x_1'=\delta(x_1+2ax_2+3a^2x_3+\ldots)\\ x_2'=\delta^2(x_2+3ax_3+\ldots)\\\ldots\ldots\ldots\\x_n'=\delta^nx_n
\end{cases}.$$
In particular
$$\|E_Hg\|_{L^p(B_R)}=\delta^{1-\frac{n(n+1)}{2p}}\|E_{[0,1]}g^{a,\delta}\|_{L^{p}(C_R)}$$
where $C_R$ is a $\sim \delta R\times \delta^2 R\times \ldots \times \delta^n R$ - cylinder. Cover $C_R$ by a family $\F$ of  $O(1)$-overlapping $\delta^nR$- balls $B_{\delta^nR}$ and write using Proposition \ref{p3} and Minkowski's inequality
$$\|E_Hg\|_{L^{p}(B_R)}\lesssim \delta^{1-\frac{n(n+1)}{2p}}(\sum_{B_{\delta^nR}\in \F}\|E_{[0,1]}g^{a,\delta}\|_{L^p(B_{\delta^nR})}^{p})^{1/p}\le$$

$$C_{\nu}C_p(\delta^nR,\nu)\delta^{1-\frac{n(n+1)}{2p}}(\sum_{B_{\delta^nR}\in \F}(\sum_{|U|=(\delta^nR)^{-1/n}}\|E_{U}g^{a,\delta}\|_{L^{p}(w_{B_{\delta^n R}})}^{2})^{p/2})^{1/p}+$$

$$+C\delta^{1-\frac{n(n+1)}{2p}}(\sum_{B_{\delta^nR}\in \F}(\sum_{|I|=\nu}\|E_{I}g^{a,\delta}\|_{L^{p}(B_{\delta^n R})}^{2})^{p/2})^{1/p}\lesssim$$
$$C_{\nu}C_p(\delta^nR,\nu)\delta^{1-\frac{n(n+1)}{2p}}(\sum_{|U|=(\delta^nR)^{-1/n}}\|E_{U}g^{a,\delta}\|_{L^{p}(w_{C_R})}^{2})^{1/2}+$$
$$+C\delta^{1-\frac{n(n+1)}{2p}}(\sum_{|I|=\nu}\|E_{I}g^{a,\delta}\|_{L^{p}(w_{C_R})}^{2})^{1/2}.$$

Changing back to the original variables gives us the desired estimate.
\end{proof}
\bigskip

We are now ready to prove  Theorem \ref{wt1}. Given $\epsilon>0$, choose $\nu$ so that $\log_{\nu^{-1}}C\ll n\epsilon$, where $C$ is the constant in Proposition \ref{p4}. Iterate Proposition \ref{p4} starting with scale $\delta=1$ until we reach scale $\delta=R^{-1/n}$.
Each iteration  lowers the scale of the intervals from $\delta$ to  ${\delta}\nu$. We thus have to iterate $\frac{\log_{\nu^{-1}} R}{n}$ times. In particular
$$\|E_{[0,1]}g\|_{L^{p}(B_R)}\lesssim\sum_{s=0}^{\frac{\log_{\nu^{-1}} R}{n}} C_{\nu}C^sC_p(R\nu^{ns},\nu)(\sum_{|U|=R^{-1/n}}\|E_{U}g\|_{L^{p}(w_{B_R})}^{2})^{1/2}$$
$$\le R^{\epsilon}C_{\nu}C_p(R,\nu)(\sum_{|U|=R^{-1/n}}\|E_{U}g\|_{L^{p}(w_{B_R})}^{2})^{1/2},$$
since $C_p(\cdot,\nu)$ is an essentially  nondecreasing function.

\bigskip

\section{The $l^2$ decoupling for curves}
To prove \eqref{e11} we develop a variant of the argument for hypersurfaces from  \cite{BD3}. That argument relied on the deep multilinear theorem of Bennett, Carbery and Tao \cite{BCT}. Our proof for curves will use the following analogous, but much simpler result, see for example Lemma 2.5 in \cite{Sang}.
\begin{proposition} Let $I_1,\ldots,I_n\subset [0,1]$ be $\nu$-transverse and let $g:[0,1]\to \C$. Then
\begin{equation}
\label{we1}
\|(\prod_{i=1}^n|E_{I_i}g|)^{1/n}\|_{L^{2n}(\R^n)}\lesssim_\nu (\prod_{i=1}^n\|g\|_{L^2(I_i)})^{1/n}
\end{equation}
\end{proposition}
 The whole argument consists of changing variables
$$(t_1,\ldots,t_n)\in I_1\times\ldots\times I_n\mapsto (\Phi(t_1)+\ldots+\Phi(t_n))\in\R^n$$
followed by Plancherel's inequality $\|\widehat{G}\|_2\lesssim \|G\|_2$. An immediate consequence is the fact that given any $f_i$ Fourier supported in a $\delta$-neighborhood of $\Phi(I_i)$, we have
\begin{equation}
\label{we2}
\|(\prod_{i=1}^n|f_i|)^{1/n}\|_{L^{2n}(\R^n)}\lesssim_\nu N^{-\frac{n-1}{2}}(\prod_{i=1}^n\|f_i\|_{L^2(\R^n)})^{1/n}.
\end{equation}
Indeed, foliate the $\delta$ neighborhood $\A_\delta$ of $\Phi([0,1])$ into translates of the curve $\Phi([0,1])$, apply \eqref{we1} to each of them and use H\"older's inequality.

Another immediate corollary of \eqref{we1} is the following multilinear version of the decoupling inequality \eqref{e11}, that follows by simply invoking $L^2$ orthogonality
\begin{equation}
\label{we10}
\|(\prod_{i=1}^n|E_{I_i}g|)^{1/n}\|_{L^{2n}(B_N)}\lesssim_\nu \|(\prod_{i=1}^n\sum_{J\subset I_i\atop{|J|=N^{-1}}}|E_{J}g|^2)^{\frac{1}{2n}}\|_{L^{2n}(w_{B_N})}
\end{equation}
There are two ways in which \eqref{we10} is stronger than \eqref{e11}. First, note that the  multilinear decoupling holds at frequency scales as small as $N^{-1}$. Second, note that by Minkowski's inequality
$$\|(\sum_J|h_J|^2)^{1/2}\|_{p}\le (\sum_J\|h_J\|_p^2)^{1/2},$$
which makes the right hand side of \eqref{we10} smaller than the right hand side of \eqref{e11}.

Using the fact that $E_{J}g$ is essentially constant on each $B_N$, then interpolating with the trivial $L^2$ bound we get that for each $2\le p\le 2n$
$$\|(\prod_{i=1}^n|E_{I_i}g|)^{1/n}\|_{L^{p}(B_N)}\lesssim_\nu[\prod_{i=1}^n (\sum_{J\subset I_i\atop{|J|=N^{-1}}}\|E_{J}g\|_{L^{p}(w_{B_N})}^2)]^{\frac{1}{2n}}.$$
Now, if we sum this over balls $B_N$ in a finitely overlapping cover of $B_{N^n}$ and then rescale,  we get 
\begin{equation}
\label{we10vrbhhybtugbtm689um896-it9v5it689itv90i}
\|(\prod_{i=1}^n|E_{I_i}g|)^{1/n}\|_{L^{p}(B_N)}\lesssim_\nu[\prod_{i=1}^n (\sum_{J\subset I_i\atop{|J|=N^{-1/n}}}\|E_{J}g\|_{L^{p}(w_{B_N})}^2)]^{\frac{1}{2n}}
\end{equation}
An application of the  Bourgain--Guth induction on scales described in the end of the previous section shows that \eqref{we10vrbhhybtugbtm689um896-it9v5it689itv90i} implies \eqref{e11} for $2\le p\le 2n$. In the remaining part of this section we will show how to bridge the gap between $p=2n$ and $p=4n-2$.

The first step in our iteration scheme is the following lemma. We will consider a partition of $\A_\delta$ into tubular regions $\tau$, each of which is a $\delta$-neighborhood of $\Phi(J)$, for some interval $J$ of length $\delta^{1/2}$. The scale $\delta^{1/2}$ we use here is much smaller than $\delta^{1/n}$, and in particular ensures that $\tau$ is essentially a $\delta^{1/2}\times \delta\times\ldots\times \delta$ cylinder. This will in turn allow for wave packet decompositions to come into play.

We will use the notation $\delta=N^{-1}$.
\begin{lemma}
Let $I_1,\ldots,I_n\subset [0,1]$ be $\nu$-transverse intervals, and assume $f_i$ are Fourier supported in $\delta$-neighborhoods of $\Phi(I_i)$. Then for each $2n\le p\le \infty$
\begin{equation}
\label{we3}
\|(\prod_{i=1}^n(\sum_\tau|f_{i,\tau}|^2)^{1/2})^{1/n}\|_{L^{p}(\R^n)}\lesssim_\nu N^{-\frac{n(n-1)}{p}+\epsilon}(\prod_{i=1}^n\sum_{\tau}\|f_{i,\tau}\|_{L^{p/n}(\R^n)}^2)^{\frac1{2n}}.
\end{equation}
\end{lemma}
\begin{proof}
Note first that \eqref{we2} implies via a standard randomization argument that
\begin{equation}
\label{we5}
\|(\prod_{i=1}^n\sum_\tau|f_{i,\tau}|^2)^{\frac1{2n}}\|_{L^{2n}(\R^n)}\lesssim_\nu N^{-\frac{n-1}{2}}(\prod_{i=1}^n\sum_{\tau}\|f_{i,\tau}\|_{L^{2}(\R^n)}^2)^{\frac1{2n}}.
\end{equation}
We also have the trivial inequality
\begin{equation}
\label{we6}
\|(\prod_{i=1}^n\sum_\tau|f_{i,\tau}|^2)^{\frac1{2n}}\|_{L^{\infty}(\R^n)}\lesssim (\prod_{i=1}^n\sum_{\tau}\|f_{i,\tau}\|_{L^{\infty}(\R^n)}^2)^{\frac1{2n}}.
\end{equation}
The result now follows using interpolation via wave packet decompositions, see for example the proof of Proposition 6.2 in \cite{BD3}.
\end{proof}
\bigskip

Given $g:[0,1]\to \C$, there is an immediate reformulation of the lemma: for each $2n\le p\le \infty$ and each ball $B_N$ of radius $N=\delta^{-1}$ we have
\begin{equation}
\label{we4}
\|(\prod_{i=1}^n\sum_{J\subset I_i\atop{|J|=\delta^{1/2}}}|E_{J}g|^2)^{\frac1{2n}}\|_{L^{p}(w_{B_N})}\lesssim_\nu N^{-\frac{n(n-1)}{p}+\epsilon}(\prod_{i=1}^n\sum_{J\subset I_i\atop{|J|=\delta^{1/2}}}\|E_{J}g\|_{L^{p/n}(w_{B_N})}^2)^{\frac1{2n}}.
\end{equation}
Define $\kappa_p$ via the relation
$$\frac{n}{p}=\frac{1-\kappa_p}{2}+\frac{\kappa_p}p.$$

\begin{corollary} We have for each $R\ge N$
$$\|(\prod_{i=1}^n\sum_{J'\subset I_i\atop{|J'|=\delta^{1/4}}}|E_{J'}g|^2)^{\frac1{2n}}\|_{L^{p}(w_{B_R})}\le $$
\begin{equation}
\label{we14}
\le C\|(\prod_{i=1}^n\sum_{J\subset I_i\atop{|J|=\delta^{1/2}}}|E_{J}g|^2)^{\frac1{2n}}\|_{L^{p}(w_{B_R})}^{1-\kappa_p}(\prod_{i=1}^n\sum_{J'\subset I_i\atop{|J'|=\delta^{1/4}}}\|E_{J'}g\|_{L^{p}(w_{B_R})}^2)^{\frac{\kappa_p}{2n}},
\end{equation}
where $C$ depends only on $n,$ $\nu$ and $p$.
\end{corollary}
\begin{proof}
Let  $\Delta$ be an arbitrary ball of radius $N^{1/2}$. Write using H\"older's inequality
\begin{equation}
\label{we7}
(\sum_{|J'|=\delta^{1/4}}\|E_{J'}g\|_{L^{p/n}(w_{\Delta})}^2)^{\frac1{2}}\le (\sum_{|J'|=\delta^{1/4}}\|E_{J'}g\|_{L^{2}(w_{\Delta})}^2)^{\frac{1-\kappa_p}{2}}(\sum_{|J'|=\delta^{1/4}}\|E_{J'}g\|_{L^{p}(w_{\Delta})}^2)^{\frac{\kappa_p}{2}},
\end{equation}

The next key element in our argument is the almost orthogonality specific to $L^2$, which will allow us to pass from scale $\delta^{1/4}$ to scale $\delta^{1/2}$. Indeed, since $(E_Jg)w_{\Delta}$ are almost orthogonal for $|J|=\delta^{1/2}$, we have
$$(\sum_{|J'|=\delta^{1/4}}\|E_{J'}g\|_{L^{2}(w_{\Delta})}^2)^{1/2}\lesssim (\sum_{|J|=\delta^{1/2}}\|E_{J}g\|_{L^{2}(w_{\Delta})}^2)^{1/2}.$$
We can now rely on the fact that $|E_{J}g|$ is essentially constant on balls $\Delta'$ of radius $N^{1/2}$ to argue that
$$(\sum_{J\subset I_i\atop{|J|=\delta^{1/2}}}\|E_{J}g\|_{L^{2}(\Delta')}^2)^{\frac1{2}}\sim |\Delta'|^{1/2}(\sum_{J\subset I_i\atop{|J|=\delta^{1/2}}}|E_{J}g|^2)^{\frac1{2}}|_{\Delta'}$$
and thus
\begin{equation}
\label{we8}
(\prod_{i=1}^n\sum_{J\subset I_i\atop{|J|=\delta^{1/2}}}\|E_{J}g\|_{L^{2}(w_{\Delta})}^2)^{\frac1{2n}}\lesssim |\Delta|^{\frac12-\frac1p}\|(\prod_{i=1}^n\sum_{J\subset I_i\atop{|J|=\delta^{1/2}}}|E_{J}g|^2)^{\frac1{2n}}\|_{L^{p}(w_{\Delta})}.
\end{equation}
Combining \eqref{we4}, \eqref{we7} and \eqref{we8} we get
$$
\|(\prod_{i=1}^n\sum_{J'\subset I_i\atop{|J'|=\delta^{1/4}}}|E_{J'}g|^2)^{\frac1{2n}}\|_{L^{p}(w_{\Delta})}\lesssim_\nu \|(\prod_{i=1}^n\sum_{J\subset I_i\atop{|J|=\delta^{1/2}}}|E_{J}g|^2)^{\frac1{2n}}\|_{L^{p}(w_{\Delta})}^{1-\kappa_p}(\prod_{i=1}^n\sum_{J'\subset I_i\atop{|J'|=\delta^{1/4}}}\|E_{J'}g\|_{L^{p}(w_{\Delta})}^2)^{\frac{\kappa_p}{2n}}.
$$
Summing this up over $\Delta\subset B_R$ we get the desired inequality.
\end{proof}
\medskip

By iterating inequality \eqref{we14} we get for integers $s\ge 2$
$$\|(\prod_{i=1}^n\sum_{J_s\subset I_i\atop{|J_s|=\delta^{2^{-s}}}}|E_{J_s}g|^2)^{\frac1{2n}}\|_{L^{p}(w_{B_N})}\le $$
\begin{equation}
\label{we9}
\le C^{s-1}\|(\prod_{i=1}^n\sum_{J_1\subset I_i\atop{|J_1|=\delta^{1/2}}}|E_{J_1}g|^2)^{\frac1{2n}}\|_{L^{p}(w_{B_N})}^{(1-\kappa_p)^{s-1}}\prod_{j=2}^s(\prod_{i=1}^n\sum_{J_{j}\subset I_i\atop{|J_j|=\delta^{2^{-j}}}}\|E_{J_j}g\|_{L^{p}(w_{B_N})}^2)^{\frac{\kappa_p(1-\kappa_p)^{s-j}}{2n}}.
\end{equation}
\medskip

We next observe the following trivial consequence of the Cauchy--Schwartz inequality. The bound $N^{2^{-s-1}}$ is only optimal at $p=\infty$ and can be easily improved for $2n\le p<\infty$ by using \eqref{we10}. However this will not be necessary for our forthcoming argument.
\begin{lemma}
\label{wlem0081}
For $1\le p\le\infty$ and $s\ge 2$
$$\|(\prod_{i=1}^n|E_{I_i}g|)^{1/n}\|_{L^{p}({B_N})}\le N^{2^{-s-1}}\|(\prod_{i=1}^n\sum_{J_s\subset I_i\atop{|J_s|=\delta^{2^{-s}}}}|E_{J_s}g|^2)^{\frac1{2n}}\|_{L^{p}(B_N)}.$$

\end{lemma}

\bigskip

Generalized parabolic rescaling \eqref{fuyvt6rvrgyuyioeynbf7vynciuyyf} shows that for each $N^{-\rho}$ interval $J$ with $\rho\le \frac1n$ we have
$$\|E_Jg\|_{L^p(B_N)}\le K_p(N^{1-\rho n})(\sum_{I\subset J\atop{|I|=N^{-1/n}}}\|E_Ig\|_{L^p(w_{B_N})}^2)^{1/2}.$$
We will apply this rescaling to the terms in \eqref{we9} with $2^{-j}\le \frac1n$. There will be $O_n(1)$ terms in \eqref{we9} for which this procedure is not applicable. However, we will be content with the trivial estimates
$$(\prod_{i=1}^n\sum_{J_{j}\subset I_i\atop{|J_j|=\delta^{2^{-j}}}}\|E_{J_j}g\|_{L^{p}(w_{B_N})}^2)^{\frac1{2n}}\le N^A(\prod_{i=1}^n\sum_{U\subset I_i\atop{|U|=\delta^{1/n}}}\|E_{U}g\|_{L^{p}(w_{B_N})}^2)^{\frac{1}{2n}},$$
for $2^{-j}>\frac1n$, and
$$\|(\prod_{i=1}^n\sum_{J_1\subset I_i\atop{|J_1|=\delta^{1/2}}}|E_{J_1}g|^2)^{\frac1{2n}}\|_{L^{p}(w_{B_N})}\le N^A(\prod_{i=1}^n\sum_{U\subset I_i\atop{|U|=\delta^{1/n}}}\|E_{U}g\|_{L^{p}(w_{B_N})}^2)^{\frac{1}{2n}}.$$
The value of $A$ will not be important for our computations.

Combining these estimates  with \eqref{we9} and Lemma \ref{wlem0081} we have for $\nu$-transverse $I_i$
$$\|(\prod_{i=1}^n|E_{I_i}g|)^{1/n}\|_{L^{p}({B_N})}\lesssim_\nu C^{s-1}(\prod_{i=1}^n\sum_{U\subset I_i\atop{|U|=\delta^{1/n}}}\|E_{U}g\|_{L^{p}(w_{B_N})}^2)^{\frac{1}{2n}}\times$$$$
\times N^{2^{-s-1}}K_p(N^{1-2^{-s}n})^{\kappa_p}K_p(N^{1-2^{-s+1}n})^{\kappa_p(1-\kappa_p)}\ldots K_p(N^{1-2^{-j_0}n})^{\kappa_p(1-\kappa_p)^{s-j_0}}N^{O_n((1-\kappa_p)^s)},$$
where $2^{-j_0}\le \frac1n<2^{-j_0+1}.$

This in turn implies that
$$C_p(N,\nu)\lesssim_\nu $$$$C^{s-1}N^{2^{-s-1}}K_p(N^{1-2^{-s}n})^{\kappa_p}K_p(N^{1-2^{-s+1}n})^{\kappa_p(1-\kappa_p)}\ldots K_p(N^{1-2^{-j_0}n})^{\kappa_p(1-\kappa_p)^{s-j_0}}N^{O_n((1-\kappa_p)^s)}.$$
Assume $K_p(N)\sim N^{\gamma_p}$. This can be made rigorous, see for example Section 6 in \cite{BD3}. In light of Theorem \ref{wt1} we have
\begin{equation}
\label{nvhyiugyotgtr8g89t7y8b96974895}
\gamma_p\le 2^{-s-1}+\kappa_p\gamma_p(\frac{1-(1-\kappa_p)^{s-j_0+1}}{\kappa_p}-n2^{-s}\frac{1-(2(1-\kappa_p))^{s-j_0+1}}{2\kappa_p-1})+O_n((1-\kappa_p)^s).
\end{equation}
If $p>4n-2$ then $2(1-\kappa_p)=\frac{4(n-1)}{p-2}<1$. Multiply both sides of \eqref{nvhyiugyotgtr8g89t7y8b96974895} with $2^s$, simplify the algebra  and  let $s\to\infty$ to get
$\gamma_p\frac{n\kappa_p}{2\kappa_p-1}\le \frac12$ or
$$\gamma_p\le \frac{p-4n+2}{2n(p-2n)}.$$
Let $p$ approach $4n-2$. This is of course good enough to conclude the proof of Theorem \ref{t2} for $p=4n-2$. To get the range $2\le p<4n-2$, we caution that there is no interpolation argument available in this context (see the counterexample in the last section). As explained earlier, this is due to the fact that the neighborhoods $\theta\in\P_\delta$ are not straight tubes. There is however at least one way around this. Namely, note that \eqref{we14}  holds with $\kappa_p$ replaced with $\kappa=1$ (Minkowski's inequality). Thus, we can in fact replace $\kappa_p$ with any   $\kappa_p\le \kappa\le 1.$ Fix some $2n< p<p_0=4n-2$. Using $\kappa:=\kappa_{p_0+\epsilon}$, we get as before that for each $\epsilon>0$
$$\gamma_p\le \frac{2\kappa_{p_0+\epsilon}-1}{2n\kappa_{p_0+\epsilon}}.$$
Letting $\epsilon\to 0$, we  get Theorem \ref{t2} in the range $[2n,4n-2]$. Finally, recall that the result for $2\le p\le 2n$ follows  from \eqref{we10vrbhhybtugbtm689um896-it9v5it689itv90i}.

\section{A more general perspective on decouplings for curves}
\label{sec:finall}
\bigskip
Recall the definition of the extension operator
$$E_Ig(x)=\int_Ig(t)e(tx_1+t^2x_2+\ldots+t^nx_n)dt.$$
We denote by $\|F\|_{L^p_\sharp(B)}=(\frac{1}{|B|}\int_{\R^n}|F|^pw_{B})^{1/p}$ the normalized integral. Given $2\le q\le p<\infty$, $1\le m\le n$ and $\frac1n\le \alpha\le 1$, we will say that we have $(p,q,\alpha,m)$ decoupling if
$$\|(\prod_{i=1}^m|E_{I_i}g|)^{1/m}\|_{L^{p}_\sharp(B_N)}\lesssim_\epsilon N^{\alpha(\frac12-\frac1{q})+\epsilon}(\prod_{i=1}^m\sum_{|U|=N^{-\alpha}\atop{U\subset I_i}}\|E_{U}g\|_{L^{q}_\sharp(B_N)}^q)^{\frac{1}{mq}}$$
holds true for each $g:[0,1]\to \C$ , each $N$-ball $B_N\subset \R^n$ and each $\sim 1$-transverse intervals $I_i\subset [0,1]$.

A few comments are appropriate. First, $(p,q,\alpha,m)$ decoupling implies $(p',q',\alpha,m')$ decoupling when $p'\le p$, $q\le q'$, $m'\ge m$. Also, $(p,p,\alpha,m)$ decoupling implies $(p,p,\alpha',m)$ decoupling when $\alpha'\le \alpha$. It thus follows that $(p,q,\alpha,m)$ decoupling is a bit stronger than $(p,p,\alpha,m)$ decoupling, which  was referred to in the earlier part of this paper as $m$-linear $l^p$ decoupling at frequency scale $N^{-\alpha}$. In particular, $(p,q,\alpha,m)$ decoupling implies the following estimate for exponential sums:
for each $\delta$-separated set $\Lambda$ of points on the curve $\Phi([0,1])$ and each coefficients $a_\xi\in \C$ we have
\begin{equation}
\label{ropig90580vyunmv45r0-c489rt870-910-cvhn0y4n8y589t-xi39mui89t}
(\frac{1}{|B_R|}\int_{B_R}|\prod_{i=1}^m\sum_{\xi\in\Lambda\cap \Phi(I_i)}a_\xi e(\xi\cdot x)|^{p/m}dx)^{1/p}\lesssim_\epsilon \delta^{\frac12-\frac1p-\epsilon}\|a_\xi\|_{l^p(\Lambda)},
\end{equation}
for each $\epsilon$ and each ball $B_R\subset \R^n$ of radius $R\gtrsim \delta^{-\frac1\alpha}$.

\bigskip

We have repeatedly used throughout the paper the fact that one can interpolate decouplings when $\alpha=\frac12$. Let us now see why interpolation fails when $\alpha=\frac1n$. Recall that we have both $(4n-2,4n-2, \frac1n,n)$ and $(2n,2,\frac1n,n)$ decoupling. If interpolation held true, this would give $(3n,6,\frac1n,n)$ decoupling, namely
$$\|(\prod_{i=1}^n|E_{I_i}g|)^{1/n}\|_{L^{3n}_\sharp(B_N)}\lesssim_\epsilon N^{\frac1{3n}+\epsilon}(\prod_{i=1}^n\sum_{|U|=N^{-\frac1n}\atop{U\subset I_i}}\|E_{U}g\|_{L^{6}_\sharp(B_N)}^6)^{\frac{1}{6n}}.$$
Note however that this is false for $n$ large enough. Indeed, when $g=1_{[0,1]}$ the left hand side is greater than $$N^{-1/3}\|(\prod_{i=1}^n|E_{I_i}g|)^{1/n}\|_{L^{3n}(B_1)}\gtrsim N^{-1/3}.$$
On the other hand, a stationary phase computation (see for example \cite{BGGIST}) shows that if $n\ge 3$ then
$$\|E_{U}g\|_{L^{6}_\sharp(B_N)}\lesssim \|E_{[0,1]}g\|_{L^{6}_\sharp(B_N)}\lesssim_\epsilon N^{-\frac49+\epsilon}.$$
\bigskip

Given $\frac 1n\le \alpha\le 1$, let $l=l(\alpha)$ be such that $\frac1{l+1}<\alpha\le \frac1l$. If $\alpha<\frac1n$ we set $l(\alpha)=n$. The following simple example shows that
 there can be no  $(p,q,\alpha,m)$ decoupling for $p>2\frac{n-l}{\alpha}+l(l+1)$. Indeed, such a decoupling would imply that
$$(\frac1{N^n}\int_{[0,N]^n}|\prod_{i=1}^m\sum_{\frac{k}{N^\alpha}\in I_i}e(\frac{k}{N^\alpha}x_1+\ldots+(\frac{k}{N^\alpha})^nx_n)|^{p/m}dx_1\ldots dx_n)^{\frac1p}\lesssim_\epsilon N^{\frac\alpha2+\epsilon}.$$
To see why this leads to a contradiction, consider integers $M_1,\ldots,M_{l}\ge 1$ such that $M_jN^{j\alpha}\le N<(M_j+1)N^{j\alpha}$ for $1\le j\le l$.

By periodicity, the left hand  side is greater than
$$(\frac1{N^{n}}\int_{\prod_{j=1}^l[0, M_jN^{j\alpha}]\times [0,N]^{n-l}}|\prod_{i=1}^m\sum_{\frac{k}{N^\alpha}\in I_i}e(\frac{k}{N^\alpha}x_1+\ldots+(\frac{k}{N^\alpha})^nx_n)|^{p/m}dx_1\ldots dx_n)^{\frac1p}=$$
$$(\frac{\prod_{j=1}^lM_j}{N^{n}}\int_{\prod_{j=1}^l[0, N^{j\alpha}]\times [0,N]^{n-l}}|\prod_{i=1}^m\sum_{\frac{k}{N^\alpha}\in I_i}e(\frac{k}{N^\alpha}x_1+\ldots+(\frac{k}{N^\alpha})^nx_n)|^{p/m}dx_1\ldots dx_n)^{\frac1p}\gtrsim$$$$(\frac{N^l}{N^{n+\alpha+\ldots+l\alpha}}\min_{(x_1,\ldots,x_n)\in [0,\frac1{100}]^n}|\prod_{i=1}^m\sum_{\frac{k}{N^\alpha}\in I_i}e(\frac{k}{N^\alpha}x_1+\ldots+(\frac{k}{N^\alpha})^nx_n)|^{p/m})^{\frac1p}\gtrsim $$
$$N^{\alpha+\frac1{p}(l-\frac{\alpha l(l+1)}{2}-n)}.$$
\bigskip

Let us now take a look at some examples. First, the case $\alpha=1$. Inequality \eqref{we10} shows that we have $(p,2,1,n)$ decoupling for $2\le p\le 2n$. The example above shows that there can be no $(p,q,1,m)$ decoupling for $p>2n$.

The value $\alpha=\frac1n$ corresponds to the natural scaling of the curve, and it is thus naturally associated with linear decoupling ($m=1$).
Our Theorem \ref{t2} here implies that we have $(4n-2,4n-2,\frac1n,1)$ decoupling. It seems possible that $(p,p,\frac1n,1)$ decoupling would hold for $p$ as large as $n(n+1)$ and the example above shows that one can not hope for a larger $p$. If true, this would imply Vinogradov's mean value theorem.

The case $\alpha=\frac12$ is particularly interesting.  It seems reasonable to expect $(4n-2,4n-2,\frac12,m)$ decoupling for some $m>1$, at least for $m=n$. The exponent $4n-2$ is again suggested by the example above but also by interpolation heuristics.
In \cite{Bo1}, the first author proved the $(3n,6,\frac12,\frac{n}2)$ decoupling when $n$ is even. This was the key step in improving the upper  bound on the Riemann zeta function on the critical line.

As a last example for $\alpha=\frac12$, we will show below how our Theorem \ref{t1} here implies $(2(n+1),\frac{2(n+1)}{n-1},\frac12,n-1)$ decoupling. It is interesting to observe that the reciprocals of $(2n,2)$, $(3n,6)$, $(2(n+1),\frac{2(n+1)}{n-1}),(4n-2,4n-2)$ are collinear. This is consistent with the fact that interpolation  is available in the case $\alpha=\frac12$.

\begin{proposition}
\label{p2}
Let $I_1,\ldots,I_{n-1}$ be $\sim 1$-transverse intervals. Then we have
$$\|(\prod_{j=1}^{n-1}|E_{I_j}g|)^{\frac1{n-1}}\|_{L^{2(n+1)}_\sharp(B_N)}\lesssim_\epsilon N^{\frac{1}{2(n+1)}+\epsilon}(\prod_{i=1}^{n-1}\sum_{\Delta:\frac1{N^{1/2}}-\text{interval }\atop{\Delta\subset I_i}}\|E_{\Delta}g\|_{L^{\frac{2(n+1)}{n-1}}_\sharp(B_N)}^{\frac{2(n+1)}{n-1}})^{\frac{1}{2(n+1)}}.$$
\end{proposition}
\begin{proof}
The first step is to construct a hypersurface $S$ from the curve $\Phi$, and invoke Theorem \ref{t1} for it. That leads to the following inequality proved in \cite{Bo}
\begin{equation}
\label{we65}
\|(\prod_{j=1}^{n-1}|f_i|)^{\frac1{n-1}}\|_{L^{2(n+1)}(B_N)}\lesssim_\epsilon N^{\frac{1}{2(n+1)}+\epsilon}
(\sum_{\Delta_i:\frac1{N^{1/2}}-\text{interval }\atop{\Delta_i\subset I_i}}\|\prod_{i=1}^{n-1}f_{i,{\Delta_i}}\|_{L^{\frac{2(n+1)}{n-1}}(w_{B_N})}^{\frac{2(n+1)}{n-1}})^{\frac{1}{2(n+1)}},
\end{equation}
whenever $f_i$ is Fourier supported in the $\frac1N$-neighborhood of $\Phi(I_i)$ and $f_{i,{\Delta_i}}$ denotes the Fourier restriction of $f_i$ to the $\frac1N$-neighborhood of $\Phi(\Delta_i)$.
It is worth noting  that when $n=3$, the hypersurface $S$ is $$\xi_3=-\frac{\xi_1^3}{2}+\frac32\xi_1\xi_2.$$ The principal curvatures are $\sim(-\frac32(\xi_1\pm\sqrt{\xi_1^2+1}))$ and the second fundamental form of $S$ is not definite. This example explains the need for our main Theorem \ref{t1} in this paper.

The second step of the argument consists of exploiting multilinearity on the right hand side of \eqref{we65}. To avoid technicalities we will be provide an informal argument. Using wave packet decompositions, we note that for each $q$
$$|f_\Delta|^q\sim \sum_{T\cap B_N\not=\emptyset}c_TN^{-\frac12-n}1_T \text{ on }B_N,$$
where $T$ are parallel $N\times  N\ldots\times N\times N^{1/2}$-plates dual to the $\frac1N$-neighborhood of $\Phi(\Delta)$ and
$$\sum{c_T}\sim\int_{B_N}|f_\Delta|^q.$$
Using the key estimate
$$\int_{B_N}1_{T_1}\cdot\ldots\cdot 1_{T_{n-1}}\lesssim N^{\frac{n+1}{2}}$$
for transverse plates $T_1,\ldots,T_{n-1}$, we obtain (with $q=2\frac{n+1}{n-1}$)
$$\int_{B_N}\prod_{i=1}^{n-1}|f_{i,\Delta_i}|^{\frac{2(n+1)}{n-1}}\lesssim N^{-(n-\frac12)(n-1)+\frac{n+1}{2}}\prod_{i=1}^{n-1}\|f_{i,\Delta_i}\|_{L^{2\frac{n+1}{n-1}}(w_{B_N})}^{2\frac{n+1}{n-1}}.$$
This finishes the proof.

\end{proof}

\end{document}